\documentclass[twoside,11pt]{article}
\usepackage{amsmath,amsfonts}
\usepackage{amsthm}
\usepackage{graphicx}
\usepackage{color}
\usepackage{multicol}
\usepackage{epsfig}
\usepackage{amssymb}
\usepackage{enumerate}

\newtheorem{thm}{Theorem}[section]

\newtheorem{lemma}[thm]{Lemma}
\newtheorem{remark}{Remark}[section]

\newtheorem{theorem}[thm]{Theorem}
\newtheorem{proposition}[thm]{Proposition}

\begin{document}

\title{{Global bifurcation and stability of steady states for a bacterial colony model with density-suppressed motility }}

\author{ Manjun Ma    \thanks{%
Department of Mathematics, School of Sciences, Zhejiang Sci-tech University,
Hangzhou, Zhejiang, 310018, China; mjunm9@zstu.edu.cn }
\and Peng Xia  \thanks{%
Department of Mathematics, School of Sciences, Zhejiang Sci-tech University,
Hangzhou, Zhejiang, 310018, China; 624598686@qq.com
}  \and Qifeng Zhang\thanks{%
Department of Mathematics, School of Sciences, Zhejiang Sci-tech University,
Hangzhou, Zhejiang, 310018, China; zhangqifeng0504@163.com}\and Matti Vuorinen\thanks{%
Department of Mathematics and Statistics, University of Turku, FIN-20014 Turku, Finland; vuorinen@utu.fi
}}

\maketitle
\centerline{\small FILE: \jobname.tex}
\begin{quote}
\textbf{Abstract}: We investigate the structure and stability of the steady states for a bacterial colony model with density-suppressed motility. We treat the growth rate of bacteria as a bifurcation parameter to explore the local and global structure of the steady states. Relying on asymptotic analysis and the theory of Fredholm solvability, we derive the second-order approximate expression of the steady states. We analytically establish the stability criterion of the bifurcation solutions, and show that sufficiently large growth rate of bacteria leads to a stable uniform steady state. While the growth rate of bacteria is less than some certain value, there is pattern formation with the admissible wave mode. All the analytical results are corroborated by numerical simulations from different stages.


\indent \textbf{MR Subject Classification}: 35K55, 35K45, 35K57.
\end{quote}
\numberwithin{equation}{section}

\section{Introduction and preliminaries}\label{sec1}
The following nonlinear reaction-diffusion system
\begin{eqnarray}\label{model}
\left\{ \begin{array}{ccc}
\begin{aligned}
	&u_t=\Delta(r(v)u)+\sigma u(1-u), & x\in \Omega, \ t>0, \\
	&v_t = D \Delta v - v + u, & x\in \Omega, \ t>0
\end{aligned}
\end{array} \right.
\end{eqnarray}
 was first introduced in \cite{xf} to describe the dynamical behavior of the bacterial species \emph{Vibrio Fischeri}'s  colonies. Here $\Omega$ is a bounded open domain in $\mathbb{R}^n$, $n\geq 1$ is a positive integer, and $\Delta=\sum_{i=1}^n\frac{\partial^2}{\partial x_i^2}$.
The quantities $u(x,t)$ and $v(x,t)$ stand for the density of bacteria  and acy-homoserine lactone (AHL) secreted by \emph{Vibrio Fischeri}, respectively. The positive constants $\sigma$ and $D$ measure, respectively, the logistic growth rate of the bacteria and the diffusion rate of AHL. The diffusion rate of bacteria is state-dependent on $v$ modeled by the positive motility function $r(v)$. As stated in \cite{hyw}, the model (\ref{model}) can be transformed to
\begin{eqnarray}\label{model0}
\left\{ \begin{array}{ccc}
\begin{aligned}
	&u_t=\nabla\cdot (r(v)\nabla u+ur^\prime(v)\nabla v)+\sigma u(1-u), & x\in \Omega, \ t>0, \\
	&v_t = D \Delta v - v + u, & x\in \Omega, \ t>0,
\end{aligned}
\end{array} \right.
\end{eqnarray}
which is a chemotaxis model of Keller-Segel type proposed in \cite{ks1} if the cells/ bacteria do not sense the concentration between receptors, more details can be found in \cite{ks1} and \cite{ks2}.

At the present, for the model (\ref{model}), the rigorous mathematical results are very limited.  When $\sigma=0$ (namely bacteria have no growth), the existence of global solutions were obtained in \cite{ytw, cyk}, the metastability of non-constant steady states was discussed in \cite{xhtm}.
For the case where $\sigma>0$, the mechanism of stripe formation of (\ref{model}) was analyzed in \cite{xf} when $r(v)$ is a piecewise decreasing function. In \cite{js}, authors discussed the pattern solutions and their stability when the diffusion rate of $u$ has a drop at some critical AHL concentration, that is, $r(v)$ is a step function. The \emph{apriori } $L^\infty-$ bound, the global existence of classical solutions, the non-existence of pattern solutions and the numerical results of pattern formation and wave propagation were established in \cite{hyw} when the system (\ref{model}) is located in a two-dimensional bounded domain with zero Neumann boundary conditions and some conditions are imposed on the motility function $r(v)$. In \cite{mwp}, authors investigated the boundedness, existence and non-existence of non-constant positive classical solutions to the stationary problem of the model (\ref{model}), that is
\begin{eqnarray}\label{model1}
\left\{ \begin{array}{ccc}
\begin{aligned}
	&\Delta(r(v)u)+\sigma u(1-u)=0, & x\in \Omega, \\
	&D\Delta v - v + u=0, & x\in\Omega,\\
    &\nabla u\cdot \nu =\nabla v\cdot \nu = 0, &x\in \partial \Omega,
\end{aligned}
\end{array} \right.
\end{eqnarray}
where $\nu$ is the outward unit normal vector on $\partial\Omega$ and the domain $\Omega\subset \mathbb{R}^3$ is bounded and has smooth boundary. The boundary condition  means that there is no flux of either bacteria or AHL across the boundary of the domain. Under the condition that the motility function $r(v)$ satisfies
\begin{equation} \label{con1}
r(v)\in C^2([0, \infty)),  \  r(v)>0  \  \text{and} \  r^\prime(v)<0  \ \text{for} \ v\in [0, \infty),   \quad  \lim_{v\rightarrow +\infty}r(v)=0,
\end{equation}
in \cite{mwp} there is the following result:
\begin{lemma} \label{bounded1}
Let $\Omega$ be a bounded domain in $\mathbb{R}^n(1 \leq n\leq 3)$ with smooth boundary.  Then for any given constant $D_0>0$, there exists a positive constant $B>1$, which depends only on $D_0$ and $\Omega$,  such that any positive solution $(u,v)$ of (\ref{model1}) satisfies
 \begin{equation} \label{result1}
(u(x), v(x))\in\mathbb{ B}=\{(u,v):  \frac{1}{B}\leq  u, v \leq  B\} \   \    \text{for}   \   \    x\in \overline{\Omega}
 \end{equation}
provided that $D\geq D_0$. Furthermore, if $\lim\inf_{v\rightarrow \infty}r(v)v\in (r(0), \infty)$, such a constant $B$ is independent of $D_0$ and $\Omega$.
\end{lemma}
Since $r(v)\rightarrow 0$ as $v\rightarrow \infty$, by Lemma \ref{bounded1}, the case of degeneracy will not happen here. Due to the assumption $ r^\prime(v)<0 $, we call AHL concentration being of the repressive effect on bacterium motility.

 In order to further present the preliminaries, we now give some notations. Let $W^{m,p}(\Omega,\mathbb{R}^N)$ for $m\geq 1, \ 1<p<+\infty$ be the Sobolev space of $\mathbb{R}^N$- valued
functions with norm $\|\cdot\|_{m,p}$. When $p=2$, $W^{m,2}(\Omega,\mathbb{R}%
^N)$ is written as $H^m(\Omega)$. Let $L^p(\Omega) (1\leq p\leq \infty)$
denote the usual Lebesgue space in a bounded domain $\Omega \subset \mathbb{R%
}^n$ with norm $\|f\|_{p}=\Big(\int_{\Omega} |f(x)|^pdx\Big)^{1/p} $ for $%
1\leq p <\infty$ and $\|f\|_{\infty}=\mathrm{ess}\sup\limits_{x \in
\Omega}|f(x)|$. When $p\in (n, +\infty)$, $W^{1,p}(\Omega,\mathbb{R}^2)
\hookrightarrow C(\Omega,\mathbb{R}^2 )$ which is the space of $\mathbb{R}^2$-valued continuous functions.

The following properties of the negative Laplacian operator $-\Delta$ with zero Neumann boundary condition on $\Omega$ will be used later. There is a sequence of eigenvalues ${\lambda}^\infty_{i=0}$ satisfying
\begin{equation}
0=\lambda_0<\lambda_1<\lambda_2<\lambda_3<\cdot\cdot\cdot.
\label{index0}
\end{equation}
Each $\lambda_i$ has multiplicity $m_i\geq 1$. Let $\varphi_{ij}, \ i\geq 0, \  1\leq j \leq m_i$, be the normalized eigenfunctions corresponding to $\lambda_i$. Let $S(\lambda_i)$ be the eigenspace associated with $\lambda_i$ in $H^1(\Omega; \mathbb{R}^2)$. Then the set $\{\varphi_{ij}, \ i\geq 0, \ j=1,2,\cdots, \mathrm{dim}S(\lambda_i)\}$ forms a complete orthogonal  basis in $L^2(\Omega)$. Let $X=[H^1(\Omega)]^2$ and $X_{ij}=\{c\varphi_{ij}: 1\leq j\leq m_i, c\in \mathbb{R}^2\}$. Then
\begin{equation}
\mathbf{X}=\bigoplus_{i=1}^{\infty}\mathbf{X}_i, \ \quad\quad  \mathbf{X}_i=\bigoplus_{j=1}^{\mathrm{dim}S(\lambda_i)}\mathbf{X}_{ij},
\label{index1}
\end{equation}
where $\bigoplus$ denotes the direct sum of subspaces and $\mathrm{dim}S(\lambda_i)=m_i$ .

It is obvious that the system (\ref{model1}) has two constant solutions, i.e., $(u(x), v(x))\equiv(0,0)$ and $(u(x), v(x))\equiv(1,1)$ for all $x\in \Omega$. Linearizing (\ref{model}) with Neumann boundary at $(0,0)$ and $(1,1)$ respectively, by a simple computation, we know that $(0,0)$ is always unstable.  The linearized system at the point $(1,1)$ reads
\begin{eqnarray}\label{lin00}
\left\{ \begin{array}{ccc}
\begin{aligned}
	&\frac{d U}{dt}=r(1)\Delta U+r'(1)\Delta V-\sigma U, & x\in \Omega, \ t>0, \\
	&\frac{d V}{dt}=D\Delta V - V + U=0, & x\in\Omega,\ t>0,\\
    &\nabla U\cdot \nu =\nabla V\cdot \nu = 0, &x\in \partial \Omega, t>0.
\end{aligned}
\end{array} \right.
\end{eqnarray}
Let $(\varphi, \psi)e^{\rho t}$ be the solution of (\ref{lin00}). Then  the eigenvalue $\rho=\rho(\lambda_i)\stackrel{def}{=}\rho_i,$ of (\ref{lin00}) satisfies
\begin{equation}\label{eig0}
 \rho^2_i+\left[(D+r(1))\lambda_i+1+\sigma\right]\rho_i+[\sigma+r(1)\lambda_i](1+D\lambda_i)+r^\prime(1)\lambda_i=0, i=0, 1, 2, \cdot\cdot\cdot,
\end{equation}
where $\lambda_i, i=0, 1, 2, 3, \cdot\cdot\cdot$ is defined in (\ref{index0}). By the standard stability theory, for the stability/instability of the steady state $(1, 1)$ we have a
critical discriminant
\begin{equation}\label{con3}
\sigma-\left\{-\left[\frac{r^\prime(1)}{1+D\lambda_i}+r(1)\right]\lambda_i\right\}\stackrel{def}{=}\sigma-\sigma_i, \ i=1, 2, \cdot\cdot\cdot.
\end{equation}

It is easy to check that if  there is $i$ such that $\sigma_i>0$, then there must exist a positive integer $i^c$ such that
\begin{equation}\label{con4}
\sigma_i>0  \  \   \text{for} \  \   i\in [1, i^c],  \quad  \  \sigma_{i^c+j}\leq 0 \ \ \text{for} \ \ j=1, 2, 3, \cdot\cdot\cdot\infty.
\end{equation}
Note that \eqref{con4} implies that
\begin{equation}\label{ind3}
r'(1)+r(1)<0.
\end{equation}
Set
\begin{equation}\label{ccon4}
\sigma_a=\max_{1\leq i\leq i^c}\sigma_i=-\left[\frac{r^\prime(1)}{1+D\lambda_{i_a}}+r(1)\right]\lambda_{i_a}, \  \ i_a\in [1, i^c].
\end{equation}
Moreover, if we regard $\lambda_i$ as any real number and use (\ref{ind3}), then at
\begin{equation}\label{con8}
\lambda_i=\frac{1}{D}\left(\sqrt{\frac{-r^\prime(1)}{r(1)}}-1\right)
\end{equation}
the maximum of $\sigma_i$, denoted by $\sigma_c$, is attained as
\begin{equation}\label{con4c}
\sigma_c=\frac{1}{D}\left(\sqrt{-r^\prime(1)}-\sqrt{r(1)}\right)^2.
\end{equation}
It is clear that $\sigma_c\geq\sigma_a$, and that if  $i_a$ is such that (\ref{con8}) is true, then $\sigma_c=\sigma_a$.
We now have the lemma below.
\begin{lemma}\label{result2}
Suppose that (\ref{con1}) and (\ref{con4}) hold. Then, for (\ref{model}) with the zero Neumann boundary condition we have the following facts:

$(i)$  The steady state $\omega* =(1,1)$ is linearly stable if
either
$$
r'(1)+r(1)\geq 0
$$
or
$$
r'(1)+r(1)<0 \   \   \text{and}  \   \ \sigma D>-(r'(1)+r(1)).
$$

$(ii)$  $\omega* $ is unstable  if  $0<\sigma<\sigma_a$;  Usually, we call $k_a=\sqrt{\lambda_{i_a}}$ \emph{admissible wave number}.

$(iii)$  $\omega* $ is linearly stable if  $\sigma>\sigma_c$.
\end{lemma}
Naturally, we may expect the existence of non-constant steady states as the constant solutions are unstable and  figure out their structure. The purpose of this paper is to establish the existence and structure  of positive solutions of (\ref{model1}) in one dimensional space $\Omega=(0,l), l>0$ and to derive the criteria for the stability/unstability of each bifurcation branch.

Throughout this paper, by Lemmas \ref{bounded1} and \ref{result2},  we assume that both (\ref{con1}) and (\ref{con4}) are always true, and that $\sigma$  satisfies
\begin{equation}\label{bbc}
0<\sigma<\sigma_c
\end{equation}
for fixed constants $D$ and $l$.

This paper is organized as follows. In Section 2, we discuss the local and global  bifurcation
 to describe the structure of positive solutions near the bifurcation points and prove that these bifurcation curves can be prolonged as long as the parameter $\sigma$ is less than the critical value $\sigma_c$. In Section 3, we use the asymptotic analysis and the adjoint theory to derive the expression of the steady states. Then the stability/unstability criteria of the bifurcating solutions are given. Numerical simulations are carried out to demonstrate all the theoretical results in Section 4.
\section{Local and global bifurcation}\label{sec2}

With $\Omega=(0,l), l>0$ the system (\ref{model}) with Neumann boundary conditions can be rewritten as
\begin{eqnarray}\label{model223}
\left\{ \begin{array}{ccc}
\begin{aligned}
	&\frac{du}{dt}=(r(v)u)^{\prime\prime}+\sigma u(1-u), & x\in (0, l), \  t>0,\\
	&\frac{dv}{dt}=D v^{\prime\prime} - v + u, & x\in (0,l), \ t>0,\\
    &u^{\prime }(0)=u^{\prime }(l)=0,\ v^{\prime }(0)=v^{\prime }(l)=0,  & t>0,
\end{aligned}
\end{array} \right.
\end{eqnarray}
whose stationary system is (\ref{model1}) with one dimensional space, i.e.,
\begin{eqnarray}\label{model2}
\left\{ \begin{array}{ccc}
\begin{aligned}
	&(r(v)u)^{\prime\prime}+\sigma u(1-u)=0, & x\in (0, l),\\
	&D v^{\prime\prime} - v + u=0, & x\in (0,l),\\
    &u^{\prime }(0)=u^{\prime }(l)=0,\ v^{\prime }(0)=v^{\prime }(l)=0.
\end{aligned}
\end{array} \right.
\end{eqnarray}
We know that the eigenvalue problem
\begin{equation}
\left\{
\begin{array}{l}
-\varphi''(x)=\lambda\varphi(x),  \                   \quad\quad\quad\quad       x\in (0,l), \\[1mm]
 \varphi'(x)=0,   \        \quad\quad\quad\quad\quad\quad\quad      x=0, l
\end{array}
\right.  \label{laplace1}
\end{equation}
has a sequence of simple eigenvalues
\begin{equation}\label{eig1}
\lambda_j=(\pi j/l)^2, \ j=0, 1,2, \cdot\cdot\cdot
\end{equation}
and their corresponding eigenfunctions are
\begin{equation}
\varphi_j(x)=
\left\{
\begin{array}{l}
1, \                   \quad\quad\quad\quad\quad\quad\quad          j=0, \\[1mm]
\cos (\pi j x/l),   \        \quad\quad\quad     j>0.
\end{array}
\right.  \label{laplace12}
\end{equation}
Obviously, the set of eigenfunctions constitutes an orthogonal basis in $L^2(0, l)$. Let
$$
X=\{(u,v):  u, v\in C^2([0,l]), u'=v'=0  \, \, \,   {\rm at}\,  \  x= 0, \  l \},
$$
then $X$ is a Banach space with the usual $C^2$ norm,  and $Y=L^2(0,l)\times L^2(0,l)$ is a Hilbert space with the inner product
$$
(\omega_1, \omega_2)_Y=(u_1, u_2)_{L^2(0, l)}+(v_1, v_2)_{L^2(0, l)}
$$
for $\omega_1=(u_1,v_1) \in Y$, $\omega_2=(u_2,v_2)\in Y$. By expanding the second-order derivative term in the first equation, we have the system (\ref{model2}) in the form of
\begin{equation}\label{model22}
\left\{
\begin{array}{l}
r''(v)v'^2u+r'(v)v''u+2r'(v)v'u'+r(v)u''+\sigma u(1-u)=0,\ x\in (0,l), \\[1mm]
Dv^{\prime \prime }+u-v=0,\ x\in (0,l), \\[1mm]
u^{\prime }(0)=u^{\prime }(l)=0,\ v^{\prime }(0)=v^{\prime }(l)=0.%
\end{array}
\right.
\end{equation}
Define the map $P:  \Lambda \longrightarrow Y$ by
$$
P(\sigma, \omega)=\left(\begin{array}{c}
             r''(v)v'^2u+r'(v)v''u+2r'(v)v'u'+r(v)u''+\sigma u(1-u)\\
             Dv''+u-v
           \end{array}\right),
$$
where $\omega=(u, v)$, and $\Lambda=(0, \sigma_a)\times \mathbb{B}$ is a bounded set in $(0, \infty)\times X$. Hence, looking for the solutions of (\ref{model2}) is exactly equivalent to looking for the zero points of this map. Let $\omega^*=(u^*, v^*)=(1,1)$, then we have
$$
P(\sigma, \omega^*)=0   \   \text{for}  \   \sigma>0.
$$
We recall that, for a number $\alpha>0$, $(\alpha, \omega^*)$ is a bifurcation point of the equation $P=0$ with respect to the curve $(\sigma, \omega^*), \sigma>0$ if every neighborhood of $(\alpha, \omega^*)$ contains zeros of $P$ in $(0,\infty)\times X$ not lying on this curve. Then the results on local bifurcation of solutions for  (\ref{model2}) are as follows.
\begin{theorem}\label{th1}
Suppose that (\ref{bbc}) is true. If $j$ is a positive integer such that
\begin{equation}\label{con5}
\lambda_j<-\frac{r^\prime(1)+r(1)}{Dr(1)}, \ \text{and} \  \sigma_j\neq \sigma_k  \  \text{for all integers} \  k\neq j,
\end{equation}
 then $(\sigma_j, \omega^*)$ is a bifurcation point of $P=0$ with respect to the curve $(\sigma, \omega^*), \sigma>0$, where $\sigma_j$ is defined in (\ref{con3}). Furthermore, there is a one-parameter family of non-trivial solutions $\Gamma_j(\varepsilon)=(\sigma(\varepsilon), u(\varepsilon), v(\varepsilon))$ of the problem (\ref{model2}) for $|\varepsilon|$ sufficiently small, where $\sigma(\varepsilon), u(\varepsilon), v(\varepsilon)$ are continuous functions, $\sigma(0)=\sigma_j$ and
\begin{equation}\label{sol1}
u(\varepsilon)=u^*+\varepsilon a_j\varphi_j+o(\varepsilon), \  v(\varepsilon)=v^*+\varepsilon\varphi_j+o(\varepsilon),  \
a_j=1+D\lambda_j.
\end{equation}
The set of zero-points of $P$ consists of two curves $(\sigma, \omega^*)$ and $\Gamma_j(\varepsilon)$ in a neighborhood of the bifurcation point $(\sigma_j, \omega^*)$.
\end{theorem}
\begin{proof}
Fix $j$, according to Theorem 1.7 of \cite{r1}, we need to verify the following conditions:

$(1)$  \  the partial derivatives $P_\sigma,  P_\omega$, and $P_{\sigma \omega}$ exist and are continuous,

$(2)$  \  $\ker P_\omega(\sigma_j,  \omega^*)$ and $Y/R(P_\omega(\sigma_j,  \omega^*))$ are one-dimensional,

$(3)$  \  let $\ker P_\omega(\sigma_j,  \omega^*)=span \{\varphi\}$, then $P_{\sigma \omega}(\sigma_j, \omega^*)\varphi \notin R(P_\omega(\sigma_j, \omega^*))$.

Because we have
$$
P_\sigma=\left(
           \begin{array}{c}
             u(1-u) \\
             0 \\
           \end{array}
         \right),   \   P_{\sigma \omega}=\left(
                                            \begin{array}{cc}
                                              1-2u & 0 \\
                                              0 & 0 \\
                                            \end{array}
                                          \right),  \
$$
and
$$
L=P_\omega(\omega^*)=\left(
                       \begin{array}{cc}
                         r(1)\frac{\partial^2}{\partial x^2}-\sigma & r^\prime(1)\frac{\partial^2}{\partial x^2}  \\
                         1 & D\frac{\partial^2}{\partial x^2}-1 \\
                       \end{array}
                     \right),
$$
it is clear that the linear operators $P_\sigma,  P_\omega$, and $P_{\sigma \omega}$ are continuous. Condition $(1)$
is verified.

Let $\Phi=(\varphi, \psi)\in \ker L$ with $\varphi=\sum_{0\leq i\leq\infty,}{a_{i}\varphi_{i}}$ and $\psi=\sum_{0\leq i\leq\infty,}{b_{i}\varphi_{i}}$. Then we have
\begin{equation}\label{ker1}
\sum_{i=0}^\infty{\begin{pmatrix} -\lambda_ir(1)-\sigma&-\lambda_ir'(1) \\1 &-\lambda_iD-1 \end{pmatrix}\begin{pmatrix}
a_{i}\\b_{i}\end{pmatrix}\varphi_{i}}=0,
\end{equation}
which means that, by the definition of $\varphi_i$ in (\ref{laplace1}), all the coefficients must vanish, that is,
\begin{equation}
\left(
      \begin{array}{cc}
     -\lambda_ir(1)-\sigma & -\lambda_ir'(1) \\
        1 & -D\lambda_i-1  \\
      \end{array}
    \right)\left(
             \begin{array}{c}
               a_i \\
               b_i \\
             \end{array}
           \right)
    =0, \  i=0, 1, 2, \cdot\cdot\cdot, \infty.
\label{eq1}
\end{equation}
This equation has a nonzero solution provided that
\begin{equation}
\det\left(
      \begin{array}{cc}
     -\lambda_ir(1)-\sigma & -\lambda_ir'(1) \\
        1 & -D\lambda_i-1  \\
      \end{array}
    \right)=0,
\label{dt1}
\end{equation}
which holds if and only if
$$
\sigma=-\left[\frac{r'(1)}{1+D\lambda_i}+r(1)\right]\lambda_i\stackrel{def}{=}\sigma_i, \   i=0, 1, 2, \cdot\cdot\cdot, \infty.
$$
Obviously, if $i=0$, then $\sigma=0$, which is excluded by the assumption of the theorem.  In view of (\ref{con5}), the equation (\ref{dt1}) holds only for $i=j$ and
$$
\sigma=\sigma_j=-\left[\frac{r'(1)}{1+D\lambda_j}+r(1)\right]\lambda_j, \ \text{for some} j\in \{1, 2, \cdot\cdot\cdot, i^c\}.
$$
Then, by solving the equation (\ref{eq1}) with $i=j$, we have
\begin{equation}\label{ker2}
\ker L= span\{\Phi\}, \  \Phi=\left(
                              \begin{array}{c}
                                a_j\\
                                1 \\
                              \end{array}
                            \right)\varphi_j,
\end{equation}
where $a_j=1+D\lambda_j$ and $\varphi_j$ is defined as in (\ref{laplace12}). Moreover, the adjoint operator of $L$ is
$$
L^*=\left(
      \begin{array}{cc}
        r(1)\frac{\partial^2}{\partial x^2}-\sigma & 1 \\
        r^\prime(1) \frac{\partial^2}{\partial x^2} & D\frac{\partial^2}{\partial x^2}-1  \\
      \end{array}
    \right),
$$
which has
$$
\ker L^*=span \{\Phi^*\}, \  \  \Phi^*=\left(
                                       \begin{array}{c}
                                         a^*_j \\
                                         1 \\
                                       \end{array}
                                     \right)\varphi_j,
$$
where $a^*_j=-\frac{1+D\lambda_j}{r^\prime(1)\lambda_j}$. It is well known that $R(L)=(\ker L^*)^\bot$. Then the codimension of $R(L)$ is the same as $\dim \ker L^*=1$. Condition $(2)$ is thus satisfied.

We know that
$$
P_{\sigma\omega}(\sigma_j, \omega^*)\Phi=\left(
                                       \begin{array}{c}
                                         -(1+D\lambda_j)\varphi_j \\
                                         0 \\
                                       \end{array}
                                     \right)
$$
and
$$
<P_{\sigma\omega}(\sigma_j, \omega^*)\Phi, \Phi^*>_Y=\frac{(1+D\lambda_j)^2}{r^\prime(1)\lambda_j}<0.
$$
Hence, $P_{\sigma\omega}(\sigma_j, \omega^*)\Phi\notin R(L)$, and so condition $(3)$ is verified. In addition, (\ref{ker2}) implies that (\ref{sol1}) is true. The proof is completed.
\end{proof}
\begin{remark}
$\textbf{(a)}$ \ Theorem \ref{th1} and Lemma \ref{result2} show that if $\sigma_j\in (0, \sigma_c)$, then $(\sigma_j, \omega^*)$ is a bifurcation point with respect to the trivial branch $(\sigma, \omega^*)$. For given $D$ and $l$, the number of such bifurcation points is equal to the cardinality of indices $j$ such that $\sigma_j\in (0, \sigma_c)$,  see (\ref{con4}).
$\textbf{(b)}$ \  By Theorem \ref{th1}, let $\Upsilon$ be the closure of the non-trivial solution set of $P=0$, then $\Gamma_j$ is the connected component of $\Upsilon \cup \{(\sigma_j, \omega^*)\}$ to which $(\sigma_j, \omega^*)$ belongs, and in a neighborhood of the bifurcation point the curve $\Gamma_j$ is characterized by the eigenfunction $\varphi_j$. In the open interval $(0, l)$ the function $\varphi_j$ has exactly $j$ zeros, thus we call the non-constant solutions in $\Gamma_j$ \texttt{mode} $j$   steady states.
\end{remark}

We next apply the global bifurcation theory of Rabinowitz, in particular, Corollary 1.12 in \cite{r2}, and the Leray-Schauder degree for compact operators to give the information on the bifurcating curve $\Gamma_j$ far from the trivial equilibrium. Following the idea of \cite{tnj}, we first rewrite the system (\ref{model22}) as
\begin{equation}\label{mo}
\left\{
\begin{array}{l}
-u''=f(u,v),  \  \  x\in (0,l), \\[1mm]
-v''=g(u,v),  \  \  x\in (0,l), \\[1mm]
u^{\prime }(0)=u^{\prime }(l)=0,\ v^{\prime }(0)=v^{\prime }(l)=0,%
\end{array}
\right.
\end{equation}
where
$$
f(u,v)=\frac{1}{r(v)}\left[r''(v)v'^2u+2r'(v)v'u'-\frac{1}{D}r'(v)u(u-v)+\sigma u(1-u)\right], \  \   g(u,v)=\frac{1}{D}(u-v).
$$
Now shifting the constant state $\omega^*=(1, 1)$ to $\mathcal{O}=(0,0)$ by setting $\tilde{u}=u-1$,$\tilde{v}=v-1$, then (\ref{mo}) is transformed into
\begin{equation}\label{ml}
\left\{
\begin{array}{l}
-\tilde{u}''=f_0\tilde{u}+f_1\tilde{v}+\tilde{f}(\tilde{u},\tilde{v}),\ x\in (0,l),\\[1mm]
-\tilde{v}''=g_0\tilde{u}+g_1\tilde{v}+\tilde{g}(\tilde{u},\tilde{v}),\ x\in (0,l),\\[1mm]
\tilde{u}^{\prime }(0)=\tilde{u}^{\prime }(l)=0,\ \tilde{v}^{\prime }(0)=\tilde{v}^{\prime }(l)=0,%
\end{array}
\right.
\end{equation}
where $\tilde{f}$ and $\tilde{g}$ are higher-order terms of $\tilde{u}$ and $\tilde{v}$,
$$
f_0=f_u(1,1)=-\frac{r'(1)}{Dr(1)}-\frac{\sigma}{r(1)}, \  \
f_1=f_v(1,1)=\frac{r'(1)}{Dr(1)},
$$
and
$$
g_0=g_u(1,1)=\frac{1}{D},  \  \
g_1=g_v(1,1)=-\frac{1}{D}.
$$
Note that $g_1<0$, and $f_0>0$ since $\sigma<\sigma_c<\frac{-r'(1)}{D}$, see (\ref{con4c}).
  Let $F_\sigma$ and $F$ be the inverse operator of $f_0-\frac{d^2}{dx^2}$ and $-g_1-\frac{d^2}{dx^2}$, respectively; moreover, set
$$
\widetilde{\omega}=(\tilde{u}, \tilde{v}), \   M(\sigma)\widetilde{\omega}=\bigg(2f_0F_\sigma(\tilde{u})+f_1F_\sigma(\tilde{v}), g_0F(\tilde{u})\bigg), \  \text{and }  \   H(\sigma, \widetilde{\omega})=\left(F_\sigma(\tilde{f}(\tilde{u}, \tilde{v})), 0\right).
$$
Then the boundary value problem (\ref{ml}) can be rewritten in the matrix form
\begin{eqnarray}\label{map}
\widetilde{\omega}=M(\sigma)\widetilde{\omega}+H(\sigma, \widetilde{\omega})\stackrel{def}{=}T(\sigma, \widetilde{\omega}), \:\: \   M(\sigma)=\begin{pmatrix} 2f_0F_\sigma & f_1F_\sigma\\g_0F & 0 \end{pmatrix}, \   \   \widetilde{\omega}\in X.
\end{eqnarray}
Obviously, for any given $\sigma>0$ the linear operator $M(\sigma)$ is compact on $X$. On closed $\sigma$ sub-intervals of $(0, \infty)$ the operator $H(\sigma, \widetilde{\omega})$ is compact on $X$ and $H(\sigma, \widetilde{\omega})=o(\|\widetilde{\omega}\|)$ for $\widetilde{\omega}$ near zero uniformly.

The result below will play a critical role in the proof of the global bifurcation of the solutions to (\ref{model2}).
\begin{lemma}\label{lem1}
 Assume that (\ref{bbc}) and (\ref{con5}) are satisfied. Then $1$ is an eigenvalue of $M(\sigma_j)$ with algebraic multiplicity one.
\end{lemma}
\begin{proof}
Let $\Phi=(\varphi,\psi)$, \ $\varphi=\sum_{i=0}^{\infty} a_i\varphi_i,  \   \psi=\sum_{i=0}^{\infty} b_i\varphi_i$. The proof of Theorem \ref{th1} implies that
$$
\left(M(\sigma_j)-I\right)\Phi=0 \Rightarrow \left(
\begin{array}{cc}
r(1)\frac{d^2}{dx^2} -\sigma_{j} & r'(1) \frac{d^2}{dx^2}\\[2mm]
  1 & D\frac{d^2}{dx^2}-1
\end{array}
\right)\Phi=0,
$$
has one unique solution $\Phi=\left(\begin{array}{c}
1+D\lambda_j \\
1
  \end{array}\right)\varphi_j$, which shows that $1$ is an eigenvalue of $M(\sigma_j)$ with the unique eigenfunction. Thus, we have $\dim \ker (M(\sigma^j)-I)=1$. Next we will prove that the eigenvalue $1$ is simple. It is well known that the algebraic multiplicity of $1$ is equal to the dimension of the generalized null space $\bigcup_{i=1}^\infty\ker(M(\sigma_j)-I)^i$. So we need only to verify $\ker\left(M(\sigma_j)-I\right)\cap R\left(M(\sigma_j)-I\right)=\{0\}$. Let $M^*(\sigma_j)$ be the adjoint operator of $M(\sigma_j)$. If $(\varphi, \psi)\in \ker(M^*(\sigma_j)-I)$, from (\ref{map}) it follows that
\begin{equation}
  \left\{
\begin{array}{l}
2f_0^jF_\sigma(\varphi)+g_0^jF(\psi)=\varphi,
\\[2mm]
f_1^jF_\sigma(\varphi)=\psi,
\end{array}
\right.  \label{ker}
\end{equation}
where
$$
f_0^j=-\frac{r'(1)}{Dr(1)}-\frac{\sigma_j}{r(1)}, \  \  f_1^j=\frac{r'(1)}{Dr(1)}, \  \
g_0^j=\frac{1}{D}.
$$
By the definition of $F_\sigma$ and $F$, the system (\ref{ker}) can be expanded as
\begin{equation}
\left\{
\begin{array}{l}
-f_1^j\varphi''=f_\varphi\varphi+f_\psi\psi,
\\[2mm]
-\psi''=f_1^j\varphi-f_0^j\psi
\end{array}
\right.  \label{ker1}
\end{equation}
with
$$
f_\varphi=f_1^jg_1^j+2f_0^jf_1^j, \   \    f_\psi=f_1^jg_0^j-2f_0^jg_1^j-2(f_0^j)^2.
$$
Again set  $\varphi=\sum_{i=0}^{\infty} a_i\varphi_i,  \   \psi=\sum_{i=0}^{\infty} b_i\varphi_i$. By (\ref{ker1}), we get
\begin{equation}
\sum_{i=0}^{\infty}A_i^*\left(\begin{array}{c}
                             a_i \\
                             b_i
                           \end{array}\right)
\phi_i=0, \
 A_i^*=\left(
\begin{array}{cc}
 f_\phi-f_1^j\lambda_i & f_\psi\\
  f_1^j& -\lambda_i-f_0^j
\end{array}
\right).\nonumber
\end{equation}
By $\sigma\neq 0$ and (\ref{con5}), we know that $\det A_i^*=0$ if and only if $i=j$ and
$$
A_j^*=\left(\begin{array}{cc}
                                                      0 & 0 \\
                                                      f_1^j & -\lambda_j-f_0^j
                                                    \end{array}
                           \right).
$$
Therefore, $\ker(M^*(\sigma_j)-I)$ is generated by the unique element $\Phi^*=\left(\begin{array}{c}
                             f_0^j+\lambda_j \\
                             f_1^j
                           \end{array}\right)\phi_j$, and we know that $(M(\sigma_j)-I)\Phi=0$ has one unique solution (up to a constant multiple ) $\Phi=\left(\begin{array}{c}
                             -f_1^j \\
                             f_0^j\lambda_j
                             \end{array}\right)\varphi_j$, and thus $(\Phi,\Phi^*)=-2f_1^j\lambda_j\neq0$. Then
$\Phi\notin(Ker(M^*(\sigma_j)-I))^\perp=R(M(\sigma_j)-I)$, which implies that $\ker(M(\sigma_j)-I)\cap R(M(\sigma_j)-I)=\{0\}$. Thus, the desired result follows.
\end{proof}

We now state the results on the global bifurcation of the boundary value problem (\ref{model2}).
\begin{theorem}\label{th11}
 Suppose that (\ref{bbc}) and (\ref{con5}) are true. Then the projection of the bifurcation curve $\Gamma_j$ onto the $\sigma-$axis is an interval $(0, \sigma_j)$. Furthermore, the system (\ref{model2}) has at least one non-constant positive solution if $\sigma\in (0, \sigma_a)$ and $\sigma\neq \sigma_k$ for any positive integer $k$.
\end{theorem}
\begin{proof}
By the proof of Lemma \ref{lem1}, we know that the linear operator $I-M(\sigma): X \rightarrow X$ is a bijection when $\sigma\in (0, \sigma_a)\setminus \sigma_j$ and being located in a small neighborhood of $\sigma_j$. For this fixed $\sigma$, let $\mathcal{O}$ be an isolated solution of (\ref{map}). The index of this isolated zero of the map $I-\left(T(\sigma, . \right)$ is given by
$$
\mathrm{index}\left(I-T(\sigma, .), (\sigma, \mathcal{O})\right)=\mathrm{deg}\left(I-M(\sigma), \mathcal{B}, \mathcal{O}\right)=(-1)^\beta,
$$
where $\mathcal{B}$ is a sufficiently small ball centered at $\mathcal{O}$, and $\beta$ is the sum of the algebraic multiplicities of the eigenvalues of $M(\sigma)$ that are larger than $1$. For our bifurcation analysis, we necessarily verify that this index changes as the bifurcation parameter $\sigma$ crosses $\sigma_j$, i.e., for $\varepsilon>0$ sufficiently small,
\begin{equation}
\mathrm{index}\left(I-T(\sigma_j-\varepsilon, \cdot),(\sigma_j-\varepsilon, \mathcal{O}) \right)\neq \mathrm{index}\left(I-T(\sigma_j+\varepsilon), (\sigma_j+\varepsilon, \mathcal{O}) \right). \label{index}
\end{equation}
Indeed, if $\varrho$ is an eigenvalue of $M(\sigma)$ corresponding to an eigenfunction $(\varphi,\psi)$, then we have
\begin{equation}
\left\{
\begin{array}{l}
-\varrho\varphi''=(2-\varrho)f_0\varphi+f_1\psi,
\\
-\varrho\psi''=g_0\varphi+\varrho g_1\psi.
\end{array}
\right.  \nonumber
\end{equation}
Once again let $\varphi=\sum_{i=0}^\infty a_i\varphi_i$ and $\psi=\sum_{i=0}^\infty b_i\varphi_i$, then the above system can be expanded as
$$
\sum_{i=0}^{\infty}\begin{pmatrix}
 (2-\varrho)f_0-\varrho\lambda_i & f_1\\
 g_0& (g_1-\lambda_i)\varrho
\end{pmatrix}
\begin{pmatrix}
                a_i \\
                b_i
\end{pmatrix}\varphi_i=0.
$$
Then the set of eigenvalues of $M(\sigma)$ consists of all $\varrho's$ that solve the characteristic equation
\begin{equation}
 (f_0+\lambda_i)(\lambda_i-g_1)\varrho^2-2f_0(\lambda_i-g_1) \varrho-g_0f_1=0, i=0,1,2,\cdot\cdot\cdot . \label{eigen}
\end{equation}
Taking $\sigma=\sigma_j$, if $\varrho=1$ is a root of (\ref{eigen}), then it is concluded that
$$
\sigma_j=-\left[\frac{r'(1)}{(1+D\lambda_i)}+r(1)\right]\lambda_i=\sigma_i,
$$
and thus by the assumption (\ref{bbc}) we have $i=j$. Therefore, if we do not count the eigenvalues corresponding to $i=j$ in (\ref{eigen}), $M(\sigma)$ has the same number of eigenvalues which are larger than $1$ for all $\sigma$ close to $\sigma_j$, and have the same multiplicities. So we need only to consider the case of $i=j$ in (\ref{eigen}). Let $\varrho(\sigma)$ and $\overline{\varrho}(\sigma)$ be the two roots of (\ref{eigen}), then we have
$$
\varrho(\sigma_j)=1,   \      \overline{\varrho}(\sigma_j)=\frac{f_0^j-\lambda_j}{f_0^j+\lambda_j}<1.
$$
Obviously, for $\sigma$ close to $\sigma_j$,  $\overline{\varrho}(\sigma)<1$ is always true. Because $\varrho(\sigma)$ is an increasing function of $f_0^j$ and $f_0^j$ is a decreasing function in $\sigma$, the function $\varrho(\sigma)$  will increase with the decrease of $\sigma$. Thus, we have
$$
\varrho(\sigma_j-\varepsilon)>1 ,  \   \   {\varrho}(\sigma_j+\varepsilon)<1 ,
$$
which implies that $M(\sigma_j-\varepsilon)$ has exactly one more eigenvalue
  larger than $1\,,$ than  $M(\sigma_j+\varepsilon)$ does, and by using the same method as Lemma \ref{lem1} we can prove that the algebraic multiplicity of this eigenvalue is one. Hence, (\ref{index}) is verified.

Now, by  (\ref{index}) and Corollary 1.12 in \cite{r2}, we conclude that $\Gamma_j$ either meets $\partial \Lambda$ or meets $(\sigma_k, 0)$ for some $k\neq j$ and $\sigma_k>0$. It is easy to check that the system (\ref{model2}) is reflective. Thus, we can follow the idea in \cite{t1, niy} and use a reflective and periodic extension method, which is also exactly the same as that in \cite{tnj}, to show that the first alternative must occur. Then, by Lemmas \ref{bounded1} and \ref{result2}, we know that the desired results are true. The proof is completed.
\end{proof}
\section{Stability of bifurcating branches}\label{sec3}

In this section we shall study the stability of steady states $(\widehat{u}(x), \widehat{v}(x))$ bifurcating from $\omega^*=(u^*, v^*)=(1, 1)$ by using the asymptotic analysis and perturbation method. We first look for  the asymptotic expression of steady states $(\widehat{u}(x), \widehat{v}(x))$. To proceed, we let
\begin{equation}\label{sigma}
\sigma=\sigma_0+\sum_{k=1}^{\infty} \varepsilon^{k} \sigma_{k},
\end{equation}
where $\sigma_0$ needs to be ascertained afterwards, $0<\epsilon \leq 1$. Then expand $\widehat{u}(x)$ and $\widehat{v}(x)$ as power series in $\epsilon$, that is,
\begin{equation}\label{pattern}
\left\{
\begin{array}{l}
\widehat{u}=u^*+\sum_{k=1}^{\infty} \varepsilon^{k} u_{k}, \\[2mm]
\widehat{v}=v^*+\sum_{k=1}^{\infty} \varepsilon^{k} v_{k}.%
\end{array}%
\right.
\end{equation}%
Substituting (\ref{sigma}) and (\ref{pattern}) into (\ref{model22}), and expanding  $r(v), r'(v)$ and $r''(v)$ as Taylor expansion at $(1,1)$, we collect the coefficients of $O(\epsilon)$ and $O(\epsilon^2)$, respectively, and then have the following two systems
\begin{equation}
\left\{
\begin{array}{l}
r(1)u_1''-\sigma_0u_1+r'(1)v_1''=0, \  x\in (0,l), %
 \\
Dv_1''+u_1-v_1=0, \  x\in (0,l),  \\
u_1'(0)=u_1'(l)=0, \\
v_1'(0)=v_1'(l)=0,
\end{array}
\right.   \label{lin1}
\end{equation}%
and
\begin{equation}
\left\{
\begin{array}{l}
r(1)u_2''-\sigma_0u_2+r'(1)v_2''=G_1,  \    x\in (0,l),%
 \\
Dv_2''+u_2-v_2=0,   \  x\in (0,l), \\
u_2'(0)=u_2'(l)=0, \\
v_2'(0)=v_2'(l)=0,
\end{array}
\right.   \label{lin2}
\end{equation}%
where
$$
G_1=-r''(1)v_1'^2-r'(1)v_1''u_1-r''(1)v_1''v_1-2r'(1)v_1'u_1'-r'(1)v_1u_1''+\sigma_0u_1^2+\sigma_1u_1.
$$
Directly solving the system (\ref{lin1}) yields a unique non-constant solution (up to a constant multiple ) for some integer $j\in \{1, 2, \cdot\cdot\cdot, i^c\}$
\begin{equation}
\left\{
\begin{array}{l}
u_{1}=a(j)\varphi_j,\;\;a(j)=1+D\lambda_j>0, \\
{v}_{1}= \varphi_j,%
\end{array}%
\right.   \label{sol3}
\end{equation}
 as long as $\sigma_0$ is equal to
\begin{equation}\label{sol2}
-\left(\frac{r^\prime(1)}{1+D\lambda_j}+r(1)\right)\lambda_j\stackrel{def}{=}\sigma_{0}^{j},
\end{equation}
where $(\lambda_j, \varphi_j)$ is given by (\ref{eig1}) and (\ref{laplace12}). Here the uniqueness of solution indicates that $\sigma_0^j\neq \sigma_0^k$ for any integer $k\neq j$. In fact, here $\sigma_0^j$ is just  $\sigma_j$ in Section \ref{sec2}. Then we have
\begin{equation}\label{max}
\sigma_{\max}=\max_{j\in [1, i^c]}\sigma_0^j=\max_j \left\{-\left(\frac{r^\prime(1)}{1+D\lambda_j}+r(1)\right)\lambda_j,  \ j=1, 2, \cdot\cdot\cdot, i^c  \right\}=\sigma_0^{i_a}
\end{equation}
for some positive integer $i_a$, and here $\sigma_0^{i_a}$ is exactly $\sigma_a$ in Section \ref{sec2}. It is observed that $i_a$ is the wave mode maximizing $\sigma_0^j$ such that $\sigma_0^{i_a}$ is the maximum bifurcation value.  We shall call $i_a$ the \emph{ admissible wave mode} corresponding to \emph{the admissible wave number} in Section \ref{sec2}.

In order to solve (\ref{lin2}), we consider its adjoint system
\begin{equation}
\left\{
\begin{array}{l}
r(1)\overline{u}_{2}^{\prime \prime }-\left(\sigma_0%
+\frac{r'(1)}{D}\right) \overline{u}_{2}+\overline{v}_{2}=0, \ x\in (0, l), \\
D\overline{v}_{2}^{\prime \prime }+\frac{r'(1)}{D}\overline{u}_{2}-
\overline{v}_{2}=0,\;\;  x\in (0,l),\\
\overline{u}_{2}^{\prime }(0)=\overline{u}_{2}^{\prime }(l)=0, \\
\overline{v}_{2}^{\prime }(0)=\overline{v}_{2}^{\prime }(l)=0.
\end{array}
\right.   \label{lin3}
\end{equation}
This system has a non-constant solution
\begin{equation}
\left\{
\begin{array}{l}
\overline{u}_{2}=c(j)\varphi_j,\;\;c(j)=\frac{D(1+D\lambda_j)}{r'(1)}<0, \\
\overline{v}_{2}= \varphi_j.%
\end{array}%
\right.  \label{sol4}
\end{equation}
Here $\varphi_j$ is the same as in (\ref{sol3}). By the Fredholm alternative \cite{fre1}, the equation (\ref{lin2}) admits a solution if and only if
$$
\int_{0}^{l}\overline{u}_2G_{1}dx=0.
$$
Solving this equation yields
$$
\sigma_{1}=\sigma_{1}^j=\ 0.
$$
Then $G_1$ in (\ref{lin2}) can be simplified  to
\begin{equation*}
G_{1}=\frac{1}{2}\sigma_0^ja^2(j)+\bigg(\lambda_jr''(1)+2\lambda_jr'(1)a(j)+\frac{1}{2}\sigma_0a^2(j)\bigg)\cos (2\sqrt{\lambda_j}x).
\end{equation*}
By this, we can set a particular solution of (\ref{lin2}) as
\begin{equation}
\left\{
\begin{array}{l}
u_{2}=d_{1}(j)+d_{2}(j)\cos (2\sqrt{\lambda_j}x), \\[2mm]
v_{2}=d_{3}(j)+d_{4}(j)\cos (2\sqrt{\lambda_j}x).%
\end{array}%
\right.   \label{pattern2}
\end{equation}
Substitution of (\ref{pattern2}) into (\ref{lin2}) leads to
\begin{eqnarray}
d_{1}(j)&=&d_{3}(j)=-\frac{a^2(j)}{2};  \  \quad   d_2(j)=(1+4D\lambda_j)d_4(j);   \notag \\
\notag\\
d_{4}(j)&=&\frac{\lambda_jr''(1)+2\lambda_jr'(1)a(j)+\frac{1}{2}\sigma_0a^2(j)}{-4r'(1)\lambda_j+
\left(-4r(1)\lambda_j-\sigma_0^j\right)\left(1+4D\lambda_j \right)}.  \label{d1234}
\end{eqnarray}%
On account of $\sigma_1=0$, we need to find the expression of $\sigma_2$. Again substituting (\ref{sigma}) and (\ref{pattern}) into (\ref{model22}) and equating the coefficients of $O(\epsilon^3)$, we have
\begin{equation}\label{lin4}
\left\{
\begin{array}{l}
r(1)u_{3}^{\prime \prime }-\sigma_0 u_{3}+r'(1)v_{3}^{\prime \prime }=G_{2}, \  \  x\in (0, l),\\
Dv_{3}^{\prime \prime }+u_{3}- v_{3}=0, \  \   x\in (0,l), \\
u_{3}^{\prime }(0)=u_{3}^{\prime }(l)=0, \\
v_{3}^{\prime }(0)=v_{3}^{\prime }(l)=0,%
\end{array}%
\right.
\end{equation}
where
\begin{eqnarray}\label{F2}
\begin{aligned}
G_2=&-2r''(1)v_1'v_2'-r''(1)v_1'^2u_1-r'''(1)v_1v_1'^2-r'(1)v_1''u_2-r'(1)v_2''u_1\\
\ &-r''(1)v_1v_1''u_1-r''(1)v_1v_2''-r''(1)v_2v_1''-\frac{1}{2}r'''(1)v_1^2v_1''-2r'(1)v_1'u_2'\\
\ &-2r'(1)v_2'u_1'-2r''(1)v_1v_1'u_1'-r'(1)v_1u_2''-r'(1)v_2u_1''-\frac{1}{2}r''(1)v_1^2u_1''\\
\ &+2\sigma_0^ju_1u_2+\sigma_2u_1. \\
\end{aligned}
\end{eqnarray}
Then applying the solvability condition $\int_{0}^{l}\overline{u}_2G_{2}dx=0$ of (\ref{lin4}) yields
\begin{eqnarray}\label{sigma2}
\begin{aligned}
\sigma _{2}^j=& -r''(1)\lambda_j\left[\frac{3}{8}+\frac{d_3(j)}{a(j)}+\frac{d_4(j)}{2a(j)}\right]-r'(1)\lambda_j\left[\frac{d_1(j)}{a(j)}
+\frac{d_2(j)}{2a(j)}+\frac{d_4(j)}{2}+d_3(j)\right]\\
\ &-2\sigma_0^j\left[d_1(j)+\frac{1}{2}d_2(j)\right]-\frac{r'''(1)\lambda_j}{8a(j)}.\\
\end{aligned}
\end{eqnarray}
We now know that when the parameter $\sigma$ given by (\ref{sigma}) lies in the neighborhood of $\sigma_0^j$, the bifurcating solution $(\widehat{u}, \widehat{v})$ is described by (\ref{pattern}) with $(u_1, v_1)$ and  $(u_2, v_2)$ formulated by (\ref{sol3}) and (\ref{pattern2}), respectively. To bring out the relationship between the solution $(\widehat{u}, \widehat{v})$ and its bifurcation location $\sigma_0^j$, we denote
$(\widehat{u}, \widehat{v})$  by  $(\widehat{u}_j, \widehat{v}_j)$, that is,
\begin{equation}\label{pattern3}
\left\{
\begin{array}{l}
\widehat{u}_{j}=u^{\ast}+\varepsilon u_{1}+\varepsilon ^{2}u_{2}+\cdots , \\
\widehat{v}_{j}=v^{\ast}+\varepsilon v_{1}+\varepsilon
^{2}v_{2}+\cdots. %
\end{array}%
\right.
\end{equation}
Normally, $\omega^*=(u^*, v^*)$ is called the\emph{ base} term of the non-constant steady state $(\widehat{u}, \widehat{v})$ whose shape and amplitude primarily depend on the leading term $(u_1, v_1)$ when $\epsilon$ is small, that is,
\begin{equation}\label{amplitude}
\|\widehat{u}_j-u^*\|^2\approx(1+D\lambda_j)^2\frac{1}{\sigma_2^j}(\sigma-\sigma_0^j),
\end{equation}
which shows the maximum change of the bacterial density from the base term. Taking into account (\ref{laplace1}) and (\ref{sol3}), the leading term has wave mode $j$. Thus, $j$ is called the \emph{principal} wave mode of the solution $(\widehat{u}_j, \widehat{v}_j)$.

We shall analyze the stability of the solution (\ref{pattern3}) located at the $j$th bifurcating branch by discussing the sign of the principal eigenvalue of linearized system of (\ref{model22}) around $(\widehat{u}_j, \widehat{v}_j)$. Let
\begin{equation*}
\left\{
\begin{array}{c}
u=\widehat{u}_{j}+\varphi e^{\gamma t}, \\
v=\widehat{v}_{j}+\psi e^{\gamma t}%
\end{array}%
\right.
\end{equation*}
 and substitute it into (\ref{model22}), and expand  $r(v), r'(v)$ and $r''(v)$ as Taylor expansion at $(\widehat{u}_j, \widehat{v}_j)$. Then the linearized system of (\ref{model22}) is
 \begin{equation}\label{lin5}
 \left\{
\begin{array}{l}
r(\widehat{v}_j)\varphi^{\prime \prime }+r'(\widehat{v}_{j})\widehat{u}_j\psi^{\prime\prime}+2r'(\widehat{v}_j){\widehat{v}_{j}}^{\prime}\varphi'
+Q_1\psi^{\prime}+Q_2\psi+Q_3\varphi=\gamma \varphi,\;\;x\in (0,l), \\
D\psi^{\prime \prime }+\varphi-\psi=\gamma\psi,\;\;x\in (0,l), \\
\varphi^{\prime }=\varphi^{\prime }(l)=0, \\
\psi^{\prime }(0)=\psi^{\prime }(l)=0,%
\end{array}%
\right.
\end{equation}
where
\begin{flalign}
Q_1 &=2r''(\widehat{v}_{j}){\widehat{v}_{j}}^{\prime}\widehat{u}_{j}+2r'(\widehat{v}_{j}){\widehat{u}_{j}}^{\prime},\\ \nonumber
Q_2&=r'''(\widehat{v}_{j})(\widehat{v}_j^\prime)^2 \widehat{u}_{j}
+r''(\widehat{v}_{j})(\widehat{v}_{j})^{\prime\prime}\widehat{u}_{j}
+2r''(\widehat{v}_{j}){\widehat{v}_{j}}^{\prime}{\widehat{u}_{j}}^{\prime}
+r'(\widehat{v}_{j}){\widehat{u}_{j}}^{\prime\prime},\\ \nonumber
Q_3&=r''(\widehat{v}_{j})({\widehat{v}_{j}}^{\prime})^2+r'(\widehat{v}_{j}){\widehat{v}_{j}}^{\prime\prime}-\sigma \widehat{u}_{j}+\sigma(1-\widehat{u}_{j}).
\nonumber
\end{flalign}
Moreover, set
\begin{equation*}
\left\{
\begin{array}{l}
\gamma =\gamma_{0}+\epsilon \gamma_{1}+\epsilon^{2}\gamma_{2}+\cdots , \\
\varphi=\varphi _{0}+\epsilon \varphi _{1}+\epsilon^{2}\varphi_{2}+\cdots , \\
\psi=\psi_{0}+\epsilon \psi_{1}+\epsilon^{2}\psi_{2}+\cdots ,%
\end{array}%
\right.
\end{equation*}%
and
\begin{equation*}
\left\{
\begin{array}{l}
\widehat{u}_{j}=u^*+\epsilon u_{1}+\epsilon^{2}u_{2}+\cdots , \\
\widehat{v}_{j}=v^*+\epsilon v_{1}+\epsilon^{2}v_{2}+\cdots , \\
\sigma=\sigma_{0}^j+\epsilon \sigma_{1}^j+\epsilon^{2}\sigma
_{2}^j+\cdots .%
\end{array}%
\right.
\end{equation*}%
Substituting these equations into (\ref{lin5}) and equating the coefficients of $O(1)$ lead to the following system
\begin{equation}
\left\{
\begin{array}{l}
r(1)\varphi_{0}^{\prime \prime }+\frac{r'(1)}{D}(\psi_0-\varphi_0)-\sigma_{0}^j\varphi_{0}=\gamma_{0}\varphi _{0}-\frac{r'(1)}{D}\gamma_0\psi_0,\;\;x\in (0,l), \\
D\psi_{0}^{\prime \prime }+\varphi _{0}-\psi_{0}=\gamma_{0}\psi_{0},\;\;x\in (0,l), \\
\varphi _{0}^{\prime }(0)=\varphi _{0}^{\prime }(l)=0, \\
\psi_{0}^{\prime }(0)=\psi_{0}^{\prime }(l)=0.%
\end{array}%
\right.  \label{cha}
\end{equation}%
By (\ref{index1}), we  replace $(\varphi_{0}^{\prime \prime }, \psi_{0}^{\prime \prime })$ by $-\lambda_m(\varphi_{0}, \psi_{0} )$. Then the existence of a non-zero solution $(\varphi_{0}, \psi_{0} )$ yields the following equation
\begin{equation}
\gamma_{0}^{2}+\left(1+D\lambda_m+\lambda_m r(1)+\sigma_{0}^{j}\right) \gamma_{0}+E=0,
\label{gamma}
\end{equation}
where
\begin{eqnarray*}
\begin{aligned}
E=& \lambda_mr(1)\left(1+D\lambda_m \right)+\lambda_m r'(1)+\sigma_{0}^{j}\left(1+D\lambda_m \right)\\
\ &=-\sigma_{0}^{m}\left(1+D\lambda_m \right)+\sigma_{0}^{j}\left(1+D\lambda_m \right)\\
\ &=\left(1+D\lambda_m \right)\left(\sigma_{0}^{j}-\sigma_{0}^{m}\right),\\
\end{aligned}
\end{eqnarray*}
where $\sigma_0^j$ is given by (\ref{sol2}). By (\ref{max}), if $j\neq i_a$, when the positive integer $m=i_a$ such that $E<0$, then the equation (\ref{gamma}) has a positive root $\gamma_0>0$ which implies that $(\widehat{u}_j, \widehat{v}_j)$ is unstable. So we have a conclusion as follows:
\begin{proposition}\label{prop}
The non-constant steady state $(\widehat{u}_j, \widehat{v}_j)$ in (\ref{pattern3}) is unstable when $j\neq i_a$. In other words, if $(\widehat{u}_j, \widehat{v}_j)$ is stable, then it is necessary  that $j=i_a$.
\end{proposition}
We shall derive a sufficient condition for the stability of the non-constant steady states with the \emph{admissible } wave mode $i_a$. Through a simple calculation, we find that the principal eigenvalue of (\ref{cha}) is $\gamma_0=0$ with the eigenfunction
\begin{equation*}
(\varphi _{0}, \psi_{0})=((1+D\lambda_{i_a})\varphi_{i_a}, \varphi_{i_a}).
\end{equation*}
Next, we compute $\gamma_1$. Again carrying out the computation of obtaining (\ref{cha}) and equating the $O(\varepsilon)$ terms lead to
\begin{equation}
\left\{
\begin{array}{l}
r(1)\varphi _{1}^{\prime \prime }+\frac{r'(1)}{D}(\psi_1-\varphi_1)-\sigma_{0}^j\varphi_{1}=\gamma_{1}\varphi _{0}-\frac{r'(1)}{D}\gamma_1\psi_0+G_3,\;\;x\in (0,l), \\
D\psi_{1}^{\prime \prime }+\varphi _{1}-\psi_{1}=\gamma_{1}\psi_{0},\;\;x\in (0,l), \\
\varphi _{1}^{\prime }(0)=\varphi _{1}^{\prime }(l)=0, \\
\psi_{1}^{\prime }(0)=\psi_{1}^{\prime }(l)=0,%
\end{array}%
\right.  \label{cha1}
\end{equation}%
where
\begin{equation*}
\begin{aligned}
G_{3}=& -r'(1)v_{1}\varphi_{0}^{\prime \prime }-r'(1)u_{1}\psi_{0}^{\prime \prime }-r''(1)v_{1}\psi_{0}^{\prime \prime }-2r'(1)v_{1}^{\prime}\varphi_{0}^{\prime}-2r''(1)v_1'\psi_0' -2r'(1)u_1'\psi_0'\\
 \ &-r''(1)v_1''\psi_0-r'(1)u_1''\psi_0-r'(1)v_1''\varphi_0+\sigma_{1}^{i_a}\varphi_0+2\sigma_{0}^{i_a}u_1\varphi_0.\\
\end{aligned}
\end{equation*}
Applying the solvability condition of (\ref{cha1}), we have
\begin{equation}\label{eig2}
\int_{0}^{l}\left[\gamma_{1}\varphi _{0}-\frac{r'(1)}{D}\gamma_{1}\psi_{0}+G_{3}\right]\overline{u}_{2}dx+\int_{0}^{l}\gamma_{1}\psi_{0}\overline{v}_{2}dx=0,
\end{equation}%
where $(\overline{u}_{2}, \overline{v}_{2} )$ is given by (\ref{sol4}) with $j=i_a$. Solving (\ref{eig2}) for $\gamma_1$, we have
\begin{equation*}
\gamma_{1}=-\frac{\int_{0}^{l}G_{3}\overline{u}_{2}dx}{\int_{0}^{l}\left(\varphi _{0}\overline{u}_2+\psi_{0}\overline{v}_{2}-\frac{r'(1)}{D}\psi_{0}\overline{u}_{2}\right)dx}.
\end{equation*}%
A direct calculation yields
$$
\int_{0}^{l}G_{3}\overline{u}_{2}dx=0
$$
and
\begin{eqnarray} \label{eig3}
\begin{aligned}
\int_{0}^{l}\left(\varphi _{0}\overline{u}_2+\psi_{0}\overline{v}_{2}-\frac{r'(1)}{D}\psi_{0}\overline{u}_{2}\right)dx
&= \int_{0}^{l}\left(a(i_a)c(i_a)-\frac{r'(1)c(i_a)}{D}+1\right)\varphi_{i_a}^{2}dx \\
&=\frac{l}{2}\left(a(i_a)c(i_a)-\frac{r'(1)c(i_a)}{D}+1\right)\\
&=\frac{l}{2}\left(\frac{D(1+D\lambda_{i_a})^2}{r'(1)}-D\lambda_{i_a}\right)<0.\\
\end{aligned}
\end{eqnarray}
Due to $\gamma_1=0$, we need further to compute $\gamma_2$. We first simplify $G_3$ as
\begin{equation*}
G_{3}=\sigma_{0}^{i_a}a^{2}(i_{a})+\Big[ 4r'(1)a(i_a)\lambda_{i_a}+2
r''(1)\lambda_{i_a}+\sigma_{0}^{i_a}a^{2}(i_{a})\Big] \cos \Big(2\sqrt{\lambda_{i_a}}x\Big).
\end{equation*}
By this, a particular solution of (\ref{cha1}) ($\varphi_{1}, \psi_{1})$  is of the following form
\begin{equation*}
\left\{
\begin{array}{c}
\varphi_{1}=\bar{a}_{1}+\bar{a}_{2}\cos (2\sqrt{\lambda_{i_a}}x), \\
\psi_{1}=\bar{a}_{3}+\bar{a}_{4}\cos (2\sqrt{\lambda_{i_a}}x)
\end{array}
\right.
\end{equation*}
with
\begin{equation*}
\bar{a}_{i}=2d_{i}(i_{a}),  \  \    i=1, 2, 3, 4,
\end{equation*}
where $d_{i}(i_{a})$ is given by (\ref{d1234}). Again we use the same computation of obtaining (\ref{cha}), but now equate the $O(\epsilon^2)$ terms, and then get the following  system including $\gamma_2$
\begin{equation}
\left\{
\begin{array}{l}
r(1)\varphi_{2}^{\prime \prime }+\frac{r'(1)}{D}(\psi_2-\varphi_2)-\sigma_{0}^{i_a}\varphi_{2}=\gamma_{2}\varphi _{0}-\frac{r'(1)}{D}\gamma_2\psi_0+G_4,\;\;x\in (0,l), \\
D\psi_{2}^{\prime \prime }+\varphi_{2}-\psi_{2}=\gamma_{2}\psi_{0},\;\;x\in (0,l), \\
\varphi_{2}^{\prime }(0)=\varphi _{2}^{\prime }(l)=0, \\
\psi_{2}^{\prime }(0)=\psi_{2}^{\prime }(l)=0,%
\end{array}%
\right.  \label{cha2}
\end{equation}
where
\begin{eqnarray*}
\begin{aligned}
G_{4}=& -r'(1)v_{1}\varphi_{1}^{\prime\prime}-r'(1)v_2\varphi_{0}^{\prime\prime}
-\frac{1}{2}r''(1)v_{1}^{2}\varphi_{0}^{\prime\prime}-r'(1)u_1\psi_{1}^{\prime\prime}
-r'(1)u_2\psi_{0}^{\prime\prime}-r''(1)v_1\psi_{1}^{\prime\prime}\\
\ &-r''(1)v_1u_1\psi_{0}^{\prime\prime}-r''(1)v_2\psi_{0}^{\prime\prime}
-\frac{1}{2}r'''(1){v_1}^{2}\psi_{0}^{\prime\prime}-2r'(1)v_{1}^{\prime}\varphi_{1}^{\prime}
-2r'(1))v_{2}^{\prime}\varphi_{0}^{\prime}-2r''(1)v_1v_{1}^{\prime}\varphi_{0}'\\
\ &-2r''(1)v_{1}^{\prime}\psi_{1}^{\prime}
-2r''(1)v_{1}^{\prime}u_{1}\psi_{0}^{\prime}-2r''(1)v_{2}^{\prime}\psi_{0}^{\prime}
-2r'''(1)v_{1}^{\prime}v_{1}\psi_{0}^{\prime}
-2r'(1)u_{1}^{\prime}\psi_{1}^{\prime}-2r'(1)u_{2}^{\prime}\psi_{0}^{\prime}\\
\ &-2r''(1)u_{1}^{\prime}v_{1}\psi_{0}^{\prime}-r'''(1){v_{1}^{\prime}}^{2}\psi_{0}
-r''(1)v_{1}^{\prime\prime}\psi_{1}-r''(1)v_{1}^{\prime\prime}u_{1}\psi_{0}
-r''(1)v_{2}^{\prime\prime}\psi_{0}-r'''(1)v_{1}^{\prime\prime}v_{1}\psi_{0}\\
\ &-2r''(1)v_{1}^{\prime}u_{1}^{\prime}\psi_{0}-r'(1)u_{1}^{\prime\prime}\psi_{1}-r'(1)u_{2}^{\prime\prime}\psi_{0}
-r''(1)u_{1}^{\prime\prime}v_{1}\psi_{0}-r''(1){v_{1}^{\prime}}^{2}\varphi_{0}-r'(1)v_{1}^{\prime\prime}\varphi_{1}\\
\ &-r'(1)v_{2}^{\prime\prime}\varphi_{0}-r''(1)v_{1}^{\prime\prime}v_{1}\varphi_{0}+2\sigma_{0}^{i_a}u_1\varphi_1
+2\sigma_{0}^{i_a}u_2\varphi_0+\sigma_{2}^{i_a}\varphi_0.\\
\end{aligned}
\end{eqnarray*}
Again using the solvability condition of (\ref{cha2}), we have
\begin{equation}
\gamma_{2}=-\frac{\int_{0}^{l}G_{4}\overline{u}_{2}dx}{\int_{0}^{l}\left(\varphi _{0}\overline{u}_2+\psi_{0}\overline{v}_{2}-\frac{r'(1)}{D}\psi_{0}\overline{u}_{2}\right)dx}
\label{gamma2}
\end{equation}
and
$$
\int_0^lG_4\overline{u}_2dx=cl\eta,
$$
where
\begin{eqnarray*}
\begin{aligned}
\eta=& \frac{1}{4}r'(1)\lambda_{i_a}\left[6d_1+3d_2+(6d_3+3d_4)a(i_a)\right]
+\frac{1}{16}r''(1)\lambda_{i_a}\left[24d_3+12d_4+9a(i_a)\right]\\
&+\frac{3}{16}r'''(1)\lambda_{i_a}
+3\sigma_{0}^{i_a}a_{i_a}\left[d_{1}+\frac{1}{2}d_{2}\right]+\frac{1}{2}\sigma_{2}^{i_a}a(i_a).
\end{aligned}
\end{eqnarray*}
 In view of (\ref{eig3}), we know that the stability of $(\widehat{u}_{i_a}, \widehat{v}_{i_a} )$ completely depends on the sign of $\eta$, and thus we have the result below.
\begin{theorem}\label{th2}
If the positive integer $i_a$ is the admissible wave mode, then the  small-amplitude steady state $(\widehat{u}_{i_a}, \widehat{v}_{i_a} )$ of system $(\ref{model223})$ is stable provided that
\begin{equation}
\eta>0. \label{conlast}
\end{equation}
\end{theorem}
\section{Numerical simulation}\label{sec4}
This section is devoted to presenting some numerical examples to demonstrate the theoretical results obtained in Sections \ref{sec2} and \ref{sec3}. The model is solved with the MATLAB pde solver based on the finite difference scheme. For the sake of brevity, only the numerical results of the solution component $u$ are presented here. We take the length of the spatial interval $l=20$, the diffusion coefficient of AHL $D=1$ and the motility function
$$
r(v)=\frac{1}{1+e^{8(v-1)}},
$$
which obviously satisfies the condition (\ref{con1}). The small parameter $\epsilon=0.01$ is always fixed in this section.

In Section \ref{sec2}, Theorem \ref{th11} shows that each bifurcation $\Gamma_j$ emanating from $(\sigma_0^j, \omega^*), j=1, 2, \cdot\cdot\cdot, i^c$ goes backward and meets the vertical axis (i.e., $\sigma=0$), but we do not know whether  $\Gamma_j$  directly joins with $\sigma=0$ or meets some bifurcation points and then reaches $\sigma=0$; furthermore, for $\sigma=\sigma_0^j$ the existence of non-constant steady states is not established in our theorem. Therefore, we can only give the local bifurcation diagram of this example by together with Proposition \ref{prop} and Theorem \ref{th2}. Through a computation, we have $i^c=11$, $\sigma_c=0.5$, $i_a=6$ such that $\sigma_{max}=\sigma_0^6=0.4967<\sigma_c$, and $\eta=10.3042$; Moreover,
\begin{equation}\label{nu1}
\sigma_0^j>0 \  \text{for the integer } \  j\in [1, 11],  \  \sigma_0^j<0  \  \text{for }  \  j\geq 12, \ \  \text{and}  \   \   \sigma_2^j<0   \   \text{for}  \  j\in [1, 11].
\end{equation}
 Therefore, all the bifurcations are backward which clarify the results in Theorem \ref{th11}. By Proposition \ref{prop} and Theorem \ref{th2}, we have that the sixth bifurcating branch is stable and the remaining ten ones are unstable. Furthermore, all the bifurcation values can be put in order as
$$
\sigma_0^1<\sigma_0^{11}<\sigma_0^2<\sigma_0^{10}<\sigma_0^3<\sigma_0^9<\sigma_0^4<\sigma_0^8<\sigma_0^5<\sigma_0^7<\sigma_0^6.
$$
Based on the above discussion and the equation $(\ref{amplitude})$, the bifurcation diagram are presented in Figure \ref{fig1}. In order to make the diagram look cleaner, we only depict the indication of six branches, that is,  the bifurcation parameter $\sigma$ is close to the following six bifurcation points
$$
\sigma_0^6=0.4967, \sigma_0^7=0.4901, \sigma_0^8=0.4350, \sigma_0^9=0.3337,  \sigma_0^{10}=0.1895, \sigma_0^{11}=0.0054;
$$
accordingly,
$$
\sigma_{2}^{6}=-5.4569, \sigma_{2}^{7}=-8.5523, \sigma_{2}^{8}=-13.4555, \sigma_{2}^{9}=-21.4103, \sigma_{2}^{10}=-34.1442, \sigma_{2}^{11}=-54.0143.
$$
\begin{figure}[htbp]
\centering
\includegraphics[height=6cm,width=10cm]{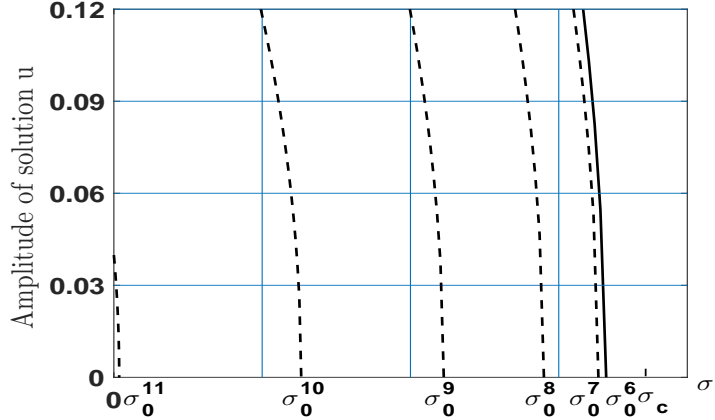} 
\caption{Local bifurcation diagram: $\sigma_{0}^{i_a}=\sigma_{max}, i_{a}=6$. The solid curve means that steady states on this branch are stable. The dotted curves show that steady states on these branches are unstable.}
\label{fig1}
\end{figure}

Next we numerically verify the asymptotic expressions of patterns in (\ref{pattern3}) and the selection mechanism of stable mode of steady states with small amplitudes (i.e., pattern solutions with small amplitudes) established in Proposition \ref{prop} and Theorem \ref{th2}. By the analytical results, the stable wave mode is $j=i_a=6$, and thus the wave number of stationary pattern is $3$, i.e., there are $3$ peaks; moreover, the second-order approximation of the pattern solution in (\ref{pattern3}) is  specified by
\begin{eqnarray}\label{pattern4}
\left\{ \begin{array}{ccc}
\begin{aligned}
&\widehat{u}_{i_a}\approx 1+\epsilon 1.8883 \cos (0.9425 x)+\varepsilon ^2(-1.7828+8.1736 \cos (1.885 x)), \\
&\widehat{v}_{i_a} \approx 1+\epsilon \cos (0.9425 x)+\epsilon^2(-1.7828+1.7952 \cos (1.885 x)),
\end{aligned}
\end{array} \right.
\end{eqnarray}
which is plotted in $(a)$ of Figure \ref{fig2}.
\begin{figure}[h]
\centering \includegraphics[height=4cm, width=6.0cm]{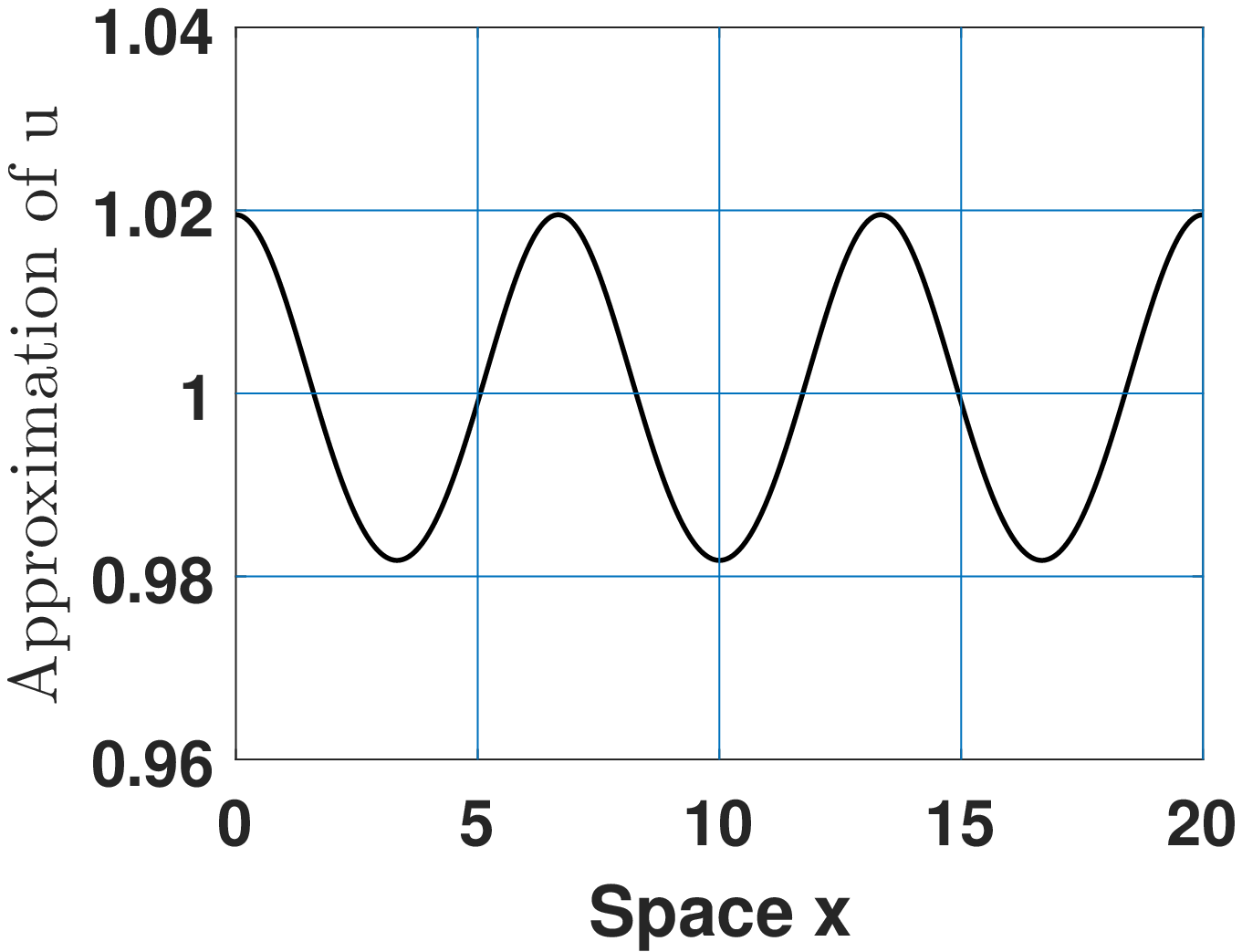} %
\includegraphics[height=4cm, width=6.0cm]{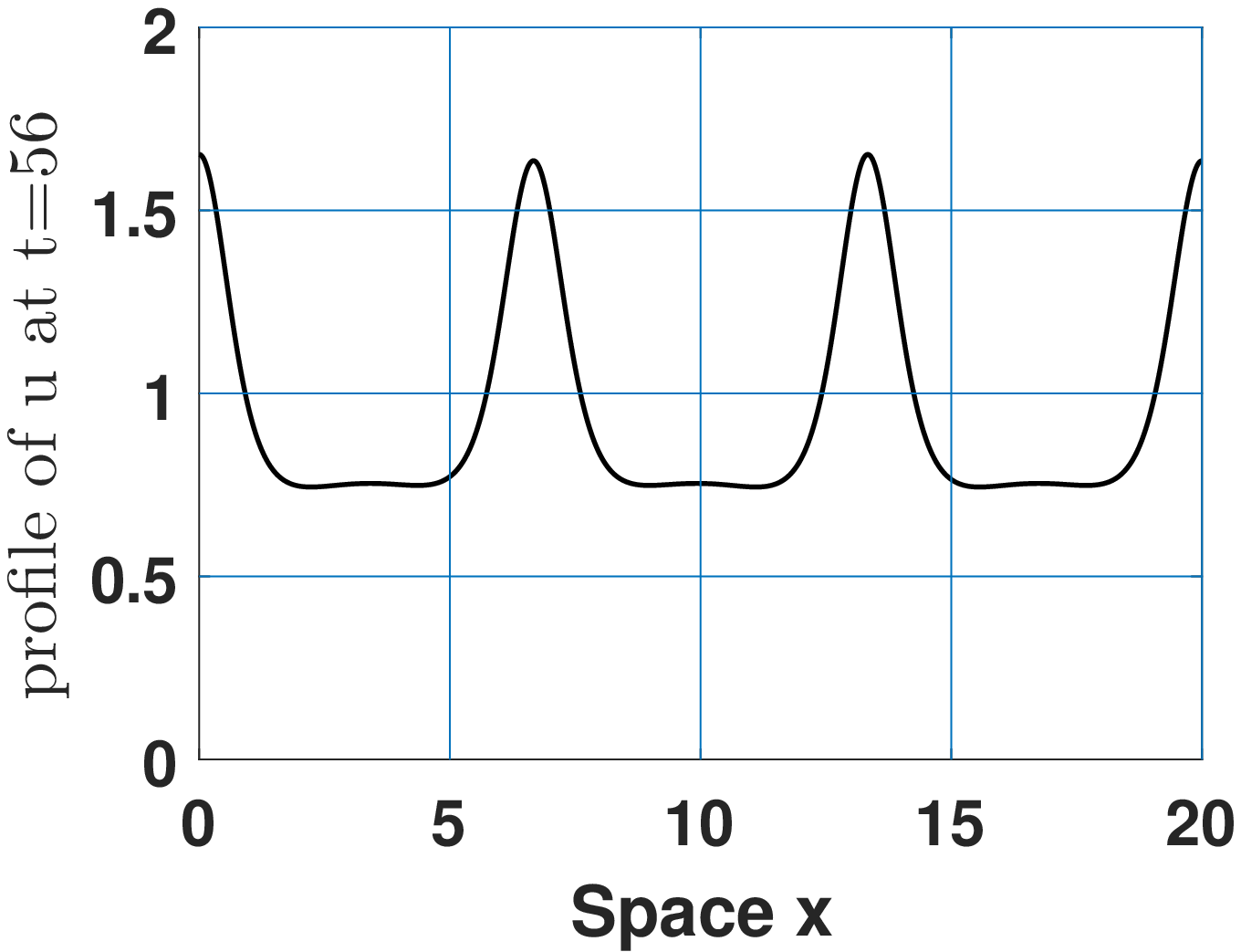} %
\hspace{5cm}(a) \hspace{7cm} (b)
\par
\par
\vspace{0.5cm}
\includegraphics[height=4cm,width=6.0cm]{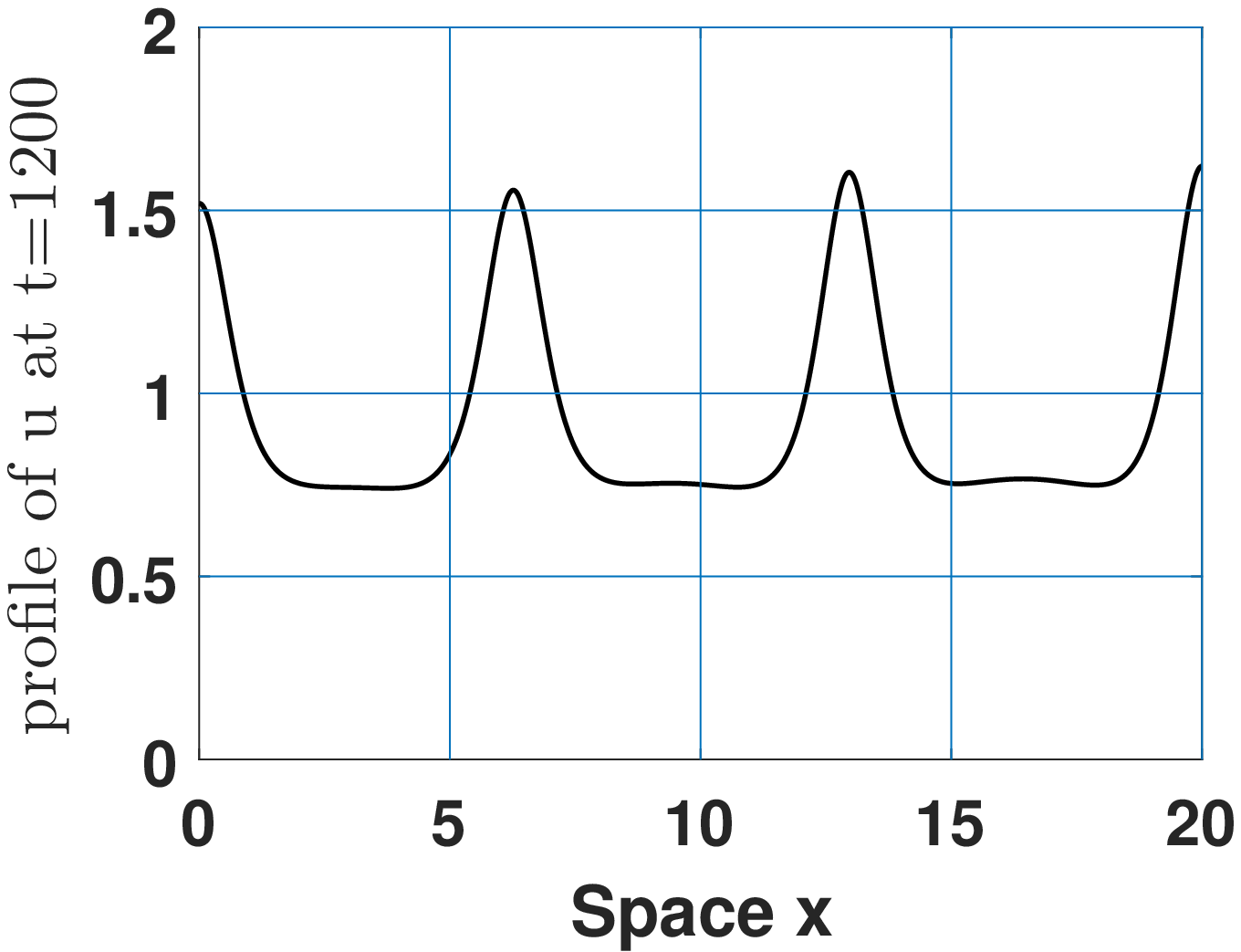} %
\includegraphics[height=4cm,width=6.0cm]{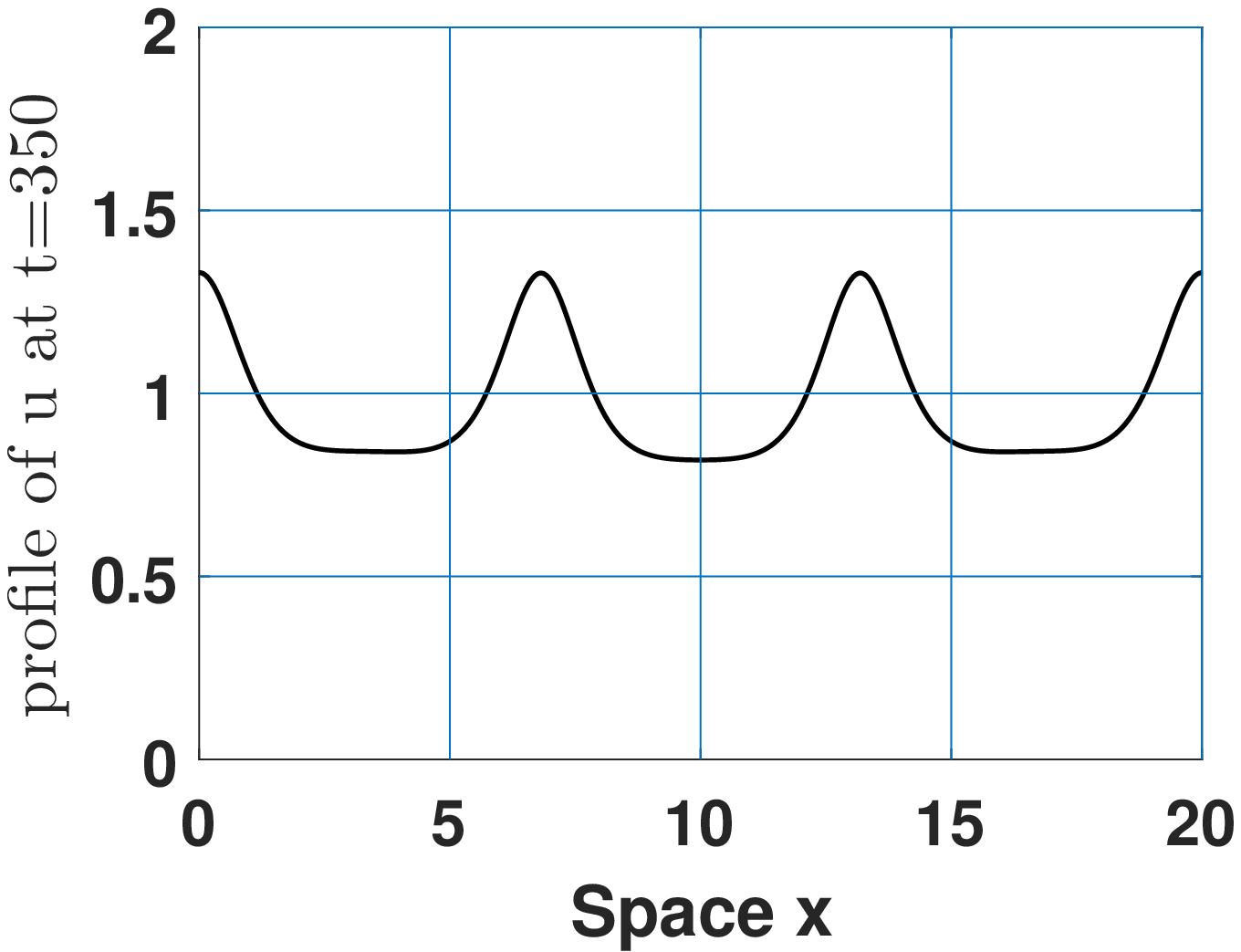} %
\hspace{5cm} (c) \hspace{7 cm} (d)
\caption{ Comparison between the asymptotically approximate solution at $O(\epsilon^3)$ and numerical solutions of (\ref{model223}) with different values of parameter $\sigma$ and different initial functions $(u_0, v_0)$ which are perturbations of the spatially homogeneous steady state $(1, 1)$.
$(a)$ The stationary state (\ref{pattern4}). The other three pattern solutions are from simulations by directly integrating the model (\ref{model}): $(b)$ $(u_0, v_0)$ is the second-order approximation of (\ref{pattern3}) with $j=3$ and $a(3)$ replaced by $1.5 a(3)$, $\sigma=0.3$; \ $(c)$  $(u_0, v_0)$ is the second-order approximation of (\ref{pattern3}) with $j=4$,  $\sigma=0.3$; \  $(d)$ $(u_0, v_0)$ is the second-order approximation of (\ref{pattern3}) with $j=4$,  $\sigma=0.4$.}
\label{fig2}
\end{figure}

In Figure \ref{fig2}, we show the comparison between the stationary state (\ref{pattern4}) predicted by the asymptotical analysis and the stationary states reached starting from different initial functions computed by numerically solving the system (\ref{model223}). It is seen that all the three simulation solutions $(b)-(d)$ are qualitatively in very good agreement with the analytical solution $(a)$, that is, all of them have three peaks (each of two boundary peaks evidently only counts as a half peak). The quantitative discrepancy is caused by ignoring the higher order terms of (\ref{pattern4}). It is also demonstrated that the bifurcation branch emanating from the maximum bifurcation point (i.e., the sixth one here) is stable since all the three simulation solutions $(b)-(d)$ tends to the steady state with wave mode $j=6$ after about running time $t=56, 1200$ and $350$, respectively. This verifies the stability criterion established in Proposition \ref{prop} and Theorem \ref{th2}.

In Figure \ref{fig3}, we show the evolution of pattern starting from different initial locations. it is also observed that the sixth bifurcation branch is stable which is precisely predicted by our analytical results in Proposition \ref{prop} and Theorem \ref{th2}. We see that solutions starting from different perturbation of the constant state $(1,1)$ finally reach the stationary pattern (i.e, the sixth bifurcation branch) which possesses 3 wave peaks (i.e., the wave mode $j=6$).
\begin{figure}[h]
\centering \includegraphics[height=6cm,width=5.0cm]{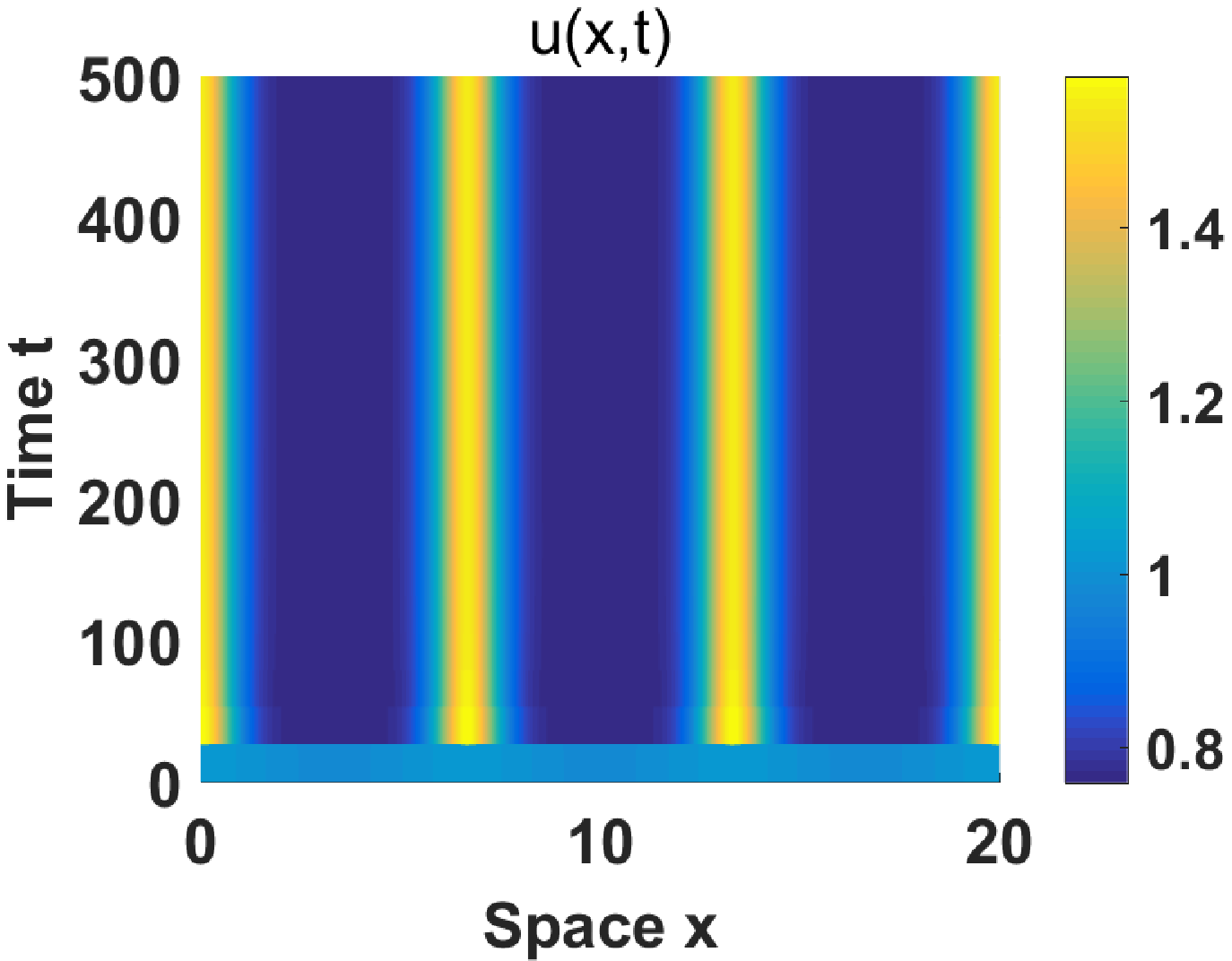} %
\includegraphics[height=6cm,width=5.0cm]{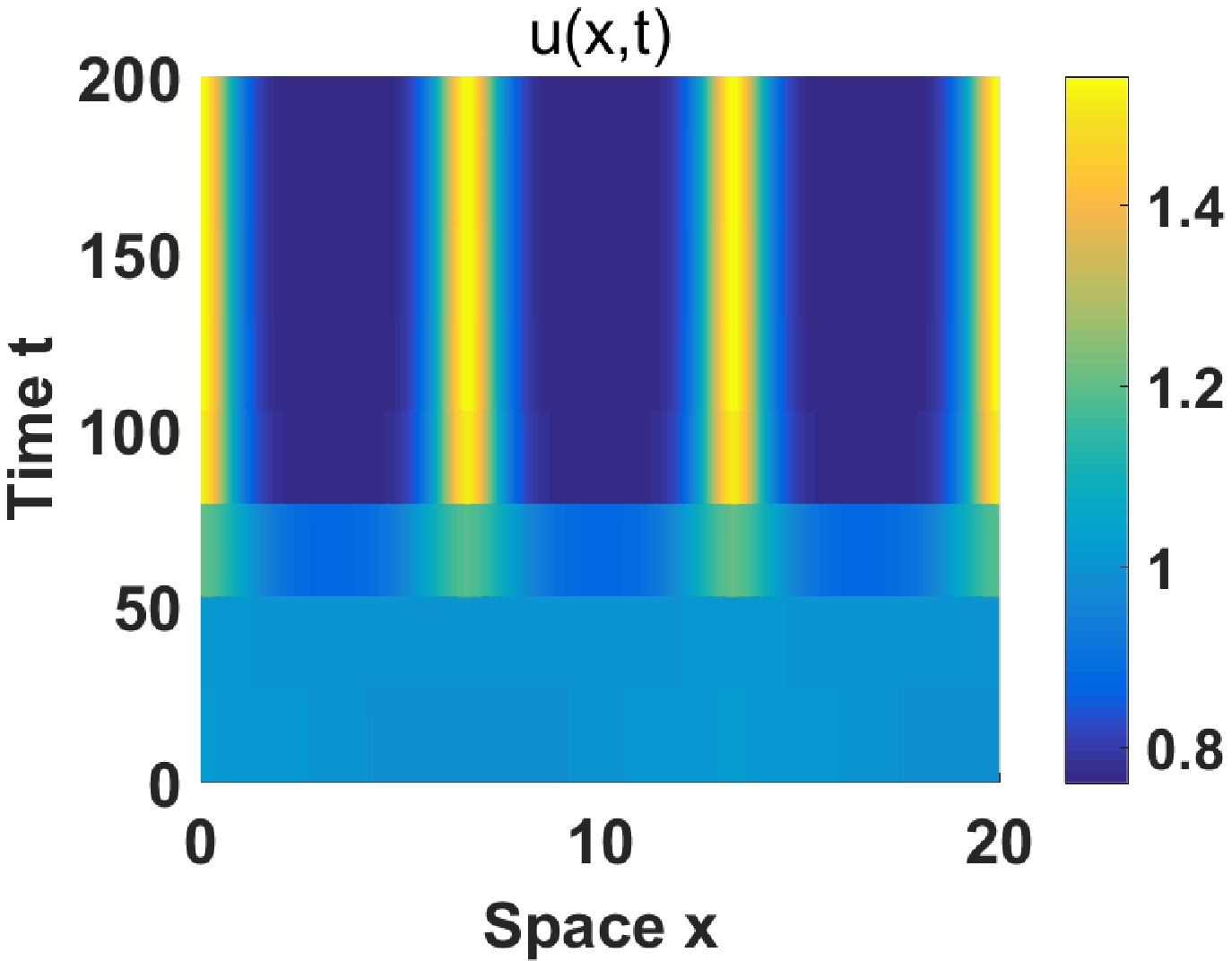} %
\includegraphics[height=6cm,width=5.0cm]{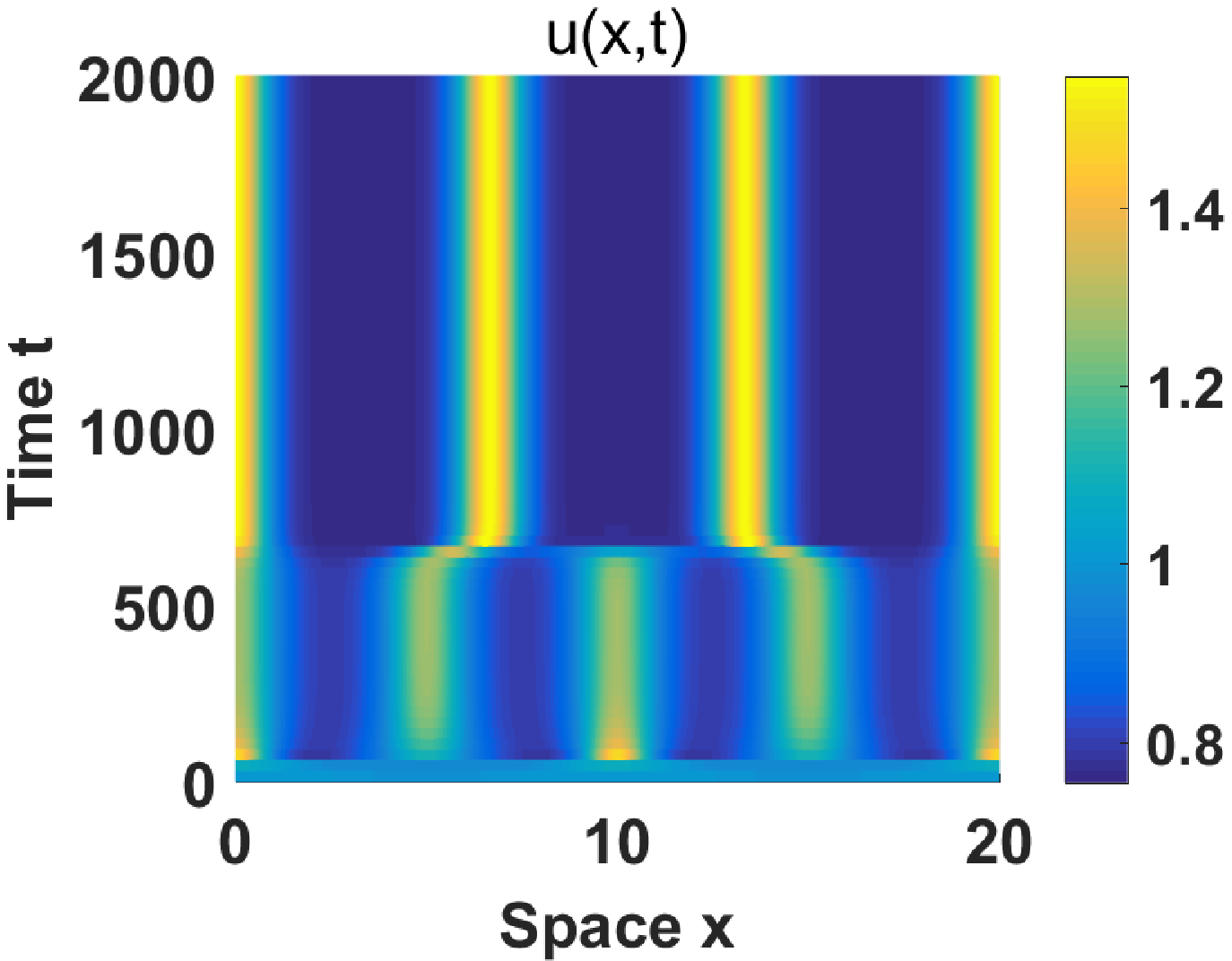} %
\hspace{1cm} (i) \hspace{4.5cm} (ii)\hspace{4.5cm} (iii)
\caption{Pattern formation of model (\ref{model223}) for  $\sigma=0.32$ and different initial functions $(u_0, v_0)$.
$(i)$ $(u_0, v_0)$ is the second-order approximation of (\ref{pattern3}) with $j=6$; $(ii)$  $(u_0, v_0)$ is the second -order approximation of (\ref{pattern3}) with $j=3$ and $a(3)$ replaced by $1.5 a(3)$; \ $(iii)$  $(u_0, v_0)$ is the second-order approximation of (\ref{pattern3}) with $j=4$ and $a(4)$ replaced by $1.2 a(4)$.}
\label{fig3}
\end{figure}

Figure \ref{fig4} displays the process of pattern formation from the initial state to the stable steady state of Figure \ref{fig3} (iii). The initial state has $2$ wave peaks (i.e, the wave mode is $j=4$), see the  left-most one in the top line. Since the mode $4$ steady state is unstable, as shown in the middle one of the top line, it is followed by the pattern having $4$ wave peaks (i.e, $j=8$) at about $t=60$. Again due to the unstability of pattern with mode $j=8$ predicted by Proposition \ref{prop}, as we see from the  left-most one in the bottom line,  at about $t=660$ merging phenomenon appears, then this solution transits to a stable state with  mode $j=6$ by reducing the $3$ wave peaks in the interior of the domain to $2$ ones. This transition is complete at about $t=750$, and then the solution keeps the state with mode $j=6$ that is stable. This once again confirms the stability criterion stated in Proposition \ref{prop} and Theorem \ref{th2}.
\begin{figure}[h]
\centering \includegraphics[height=12cm,width=17.0cm]{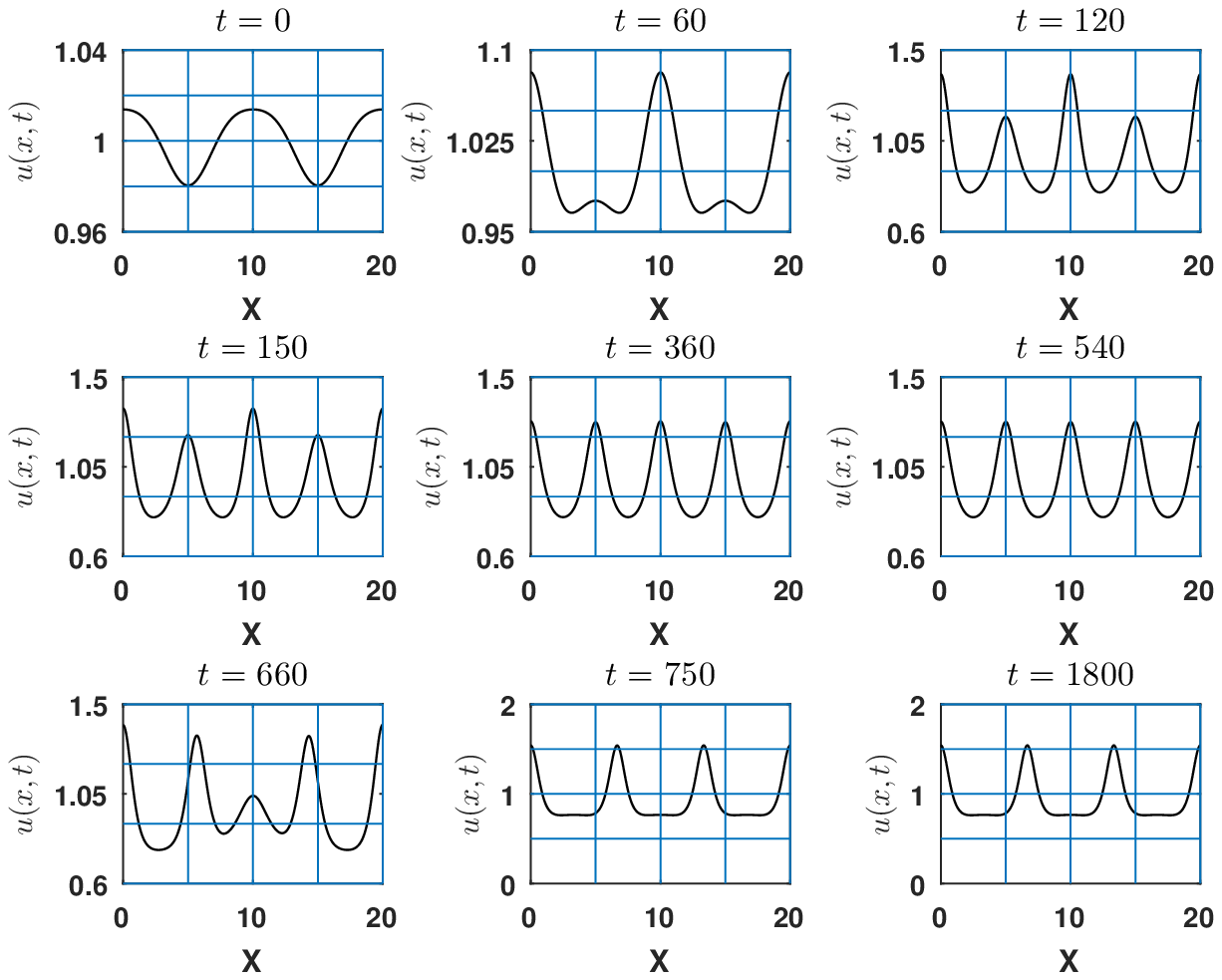} %
\caption{ Numerical simulation of the transition from the given initial data to the stable state $(\widehat{u}_6, \widehat{v}_6)$ of Figure \ref{fig3} (iii). $\sigma=0.32$, the initial function $(u_0, v_0)$ is the second-order approximation of (\ref{pattern3}) with $j=4$ and $a(4)$ replaced by $1.2 a(4)$.}
\label{fig4}
\end{figure}

\section{Conclusion}\label{sec5}
 In this work, we present the detailed information of steady states near the trivial equilibrium and the global structure of the steady states of (\ref{model223}). The second-order approximation of non-constant steady states is also derived. Then by using the standard linear stability analysis and analytical technique, we further establish the stability criterion of the bifurcation branches. It is  shown that the growth rate of bacteria $\sigma$ substantially influences the dynamical behavior of the model (\ref{model}). Particularly, under conditions (\ref{con1}) and (\ref{con4}) on the motility function $r(v)$, when the growth rate of bacteria is sufficiently large (i.e, for $\sigma>\sigma_c$), the system will keep stabilization around the uniform state $(1, 1)$; while for $\sigma \in (0, \sigma_a)$, the pattern formation must occur, and thus the densities of bacteria and AHL always depend on their location. The principal wave mode of stationary pattern coincides with the admissible one which maximizes bifurcation values. The analytical results are corroborated by direct simulations of the underlying system (\ref{model223}) through different stages.

There are various interesting questions arising from our present analytical and numerical studies. The existence of non-constant steady states and the propagation of pattern in a large domain have been rigorously discussed and organized as a separate paper. It should be noticed that the question whether the pattern formation occurs
in the cases when $\sigma\in [\sigma_{i_a}, \sigma_c]$ or $\sigma=0$ has not been discussed yet. This question may be
investigated by applying the approaches similar to that in \cite{m3, m1}. The global attractivity of non-constant steady state with the admissible wave mode still remains open. The bifurcating and emerging process in the pattern formation are numerically presented in Figures \ref{fig3} and \ref{fig4}, but the mathematical behavior of the merging process is not well understood yet. All these questions are very interesting and challenging. We think it is worthwhile to explore theories and methods of solving them in the future.

\bigbreak
\noindent \textbf{Acknowledgement}.
This work was supported  by the National Natural Science Foundation of China (No.~11671359) and the joint mobility project of the National Natural Science Foundation of China and Academy of Finland (No.~11811530145).  M. Ma gratefully acknowledges the Department of Mathematics and Statistics at University of Turku for its hospitality and the help of its members during her stay there and the financial support of Academy of Finland. M. Vuorinen's  research at the Zhejiang Sci-tech University in April 2018  was supported by the National Natural Science Foundation of China (No.~11811530145) of Prof M. Ma.



\centerline{\small FILE: \jobname.tex}


\begin{thebibliography}{99}


\bibitem{r1} M. Crandall, and P. H. Rabinowitz, \textit{Bifurcation from simple eigenvalues}, J. Functional Anal. 8 (1971) 321-340.



\bibitem{xf}  X. Fu, L.-H. Tang, C. Liu, J. -D. Huang, T. Hwa, P. Lenz, \textit{ Stripe formation in bacterial systems with density -suppressed motility}, Physical review letters 108 (19) (2012) 198102.


\bibitem{m2} G. Gambino, M.C. Lombardo, M. Sammartino, \textit{Turing instability and traveling fronts for a nonlinear reaction-diffusion system with cross-diffusion}, Math. Comput. Simulation 82 (2012) 1112-1132.


\bibitem{m3} G. Gambino, M.C. Lombardo, M. Sammartino, \textit{Pattern formation driven by cross-diffusion in 2d domain}, Nonlinear Anal. RWA 14 (2013) 1755-1779.


 \bibitem{hyw}  H. Y. Jin, Y. J. Kim, and Z. A. Wang, \textit{Boundedness, stabilization, and pattern formation driven by density- suppressed motility }, SIAM J. Appl. Math., 78(3)(2018) 1632-1657.


 \bibitem{ks1} E. F. Keller and L. A. Segel, \textit{Model for chemotaxis}, J. Theor. Biol. 30(2)(1971) 225-234.


  \bibitem{fre1}  L. Kantorovich, G. Akilov, \textit{Functional Analysis in Normed Spaces}, Macmillan, New York, 1964.


 \bibitem{cl} C. Liu, X. Fu, L. Liu, X. Ren, C. K. Chau, S. Li, L. Xiang, H. Zeng, G.Chen, L.-H. Tang, et al., \textit{Sequential establishment of stripe patterns in an expanding cell population}, Science 334 (6053) (2011) 238-241.


 \bibitem{tnj} Jaeduck Jang, Wei-Ming Ni and Moxun Tang, \textit{Global Bifurcation and Structure of Turing Patterns in the 1-D Lengyel-Epstein Model}, Journal of Dynamics and Differential Equaation 16 (2)(2004) 297-320.



 \bibitem{mwp}  M. Ma, Z. Wang, R. Peng, \textit{ Stationary and non-stationary patterns of a reaction diffusion system with density-suppressed motility}, Physica D: Nonlinear Phenomena, in press.


 \bibitem{m1} M. Ma, M. Y. Gao, C. Q. Tong, and Y. Z. Han, \textit{Chemotaxis-driven pattern formation for a reaction-diffusion-chemotaxis model with volume-filling effect}, Computers and Mathematics with Applications 72(2016) 1320-1340






\bibitem{js}  J. S. Roberge, D. Iron, and T. Kolokolnikov, \textit{ Pattern formation in bacterial colonies with density-dependent diffusion}, European Journal of Applied Mathematics, https://doi.org/10.1017/S0956792518000013Published online: 28 January 2018


\bibitem{r2} P. H. Rabinowitz, \textit{Some global results for nonlinear eigenvalue problems}, J. Functional Anal. 7 (1971) 487-513.


\bibitem{ytw} Y. Tao and M. Winkler, \textit{Effects of signal-dependent motilities in a Keller-Segel-type
reaction-diffusion system}, Math. Models Meth. Appl. Sci. 27(19)(2017) 1645-1683.


\bibitem{t1} I. Takagi, \textit{Point-condensation for a reaction-diffusion system}, J. Diff. Eqns 61(1986) 208-249.

\bibitem{xhtm} P. Xia, Y. Han, J. Tao, M. Ma, \textit{Existence and metastability of non-constant steady states in a Keller-Segel model with density-suppressed motility}, Mathematics in Applied Sciences and Engineering, in press.

 \bibitem{niy}  Nishiura. Y, \textit{Global structure of bifurcating solutions of some reaction-diffusion systems}, SIAM J. Math. Anal. 13(1982) 555-593.


\bibitem{cyk}  C. Yoon and Y.-J. Kim, \textit{Global existence and aggregation in a Keller-Segel model with
Fokker-Planck diffusion}, Acta Application Mathematics  149 (2017) 101-123.


\bibitem{ks2} C. Yoon and Y.-J. Kim, \textit{Bacterial chemotaxis without gradient sensing}, J. Math. Biol.
70(2015) 1359-1380.


%
%
%

%




\end{thebibliography}
\end{document}